\documentclass{amsart}
\usepackage{amsmath}
\usepackage{amsfonts,latexsym,amssymb}
\usepackage{amsthm}
\usepackage[sort&compress,comma,numbers]{natbib}
\usepackage[T1]{fontenc}
\usepackage{t1enc}
\usepackage[english]{babel}
\usepackage{newtxtext}          %
\usepackage{newtxmath}          %
\usepackage[utf8]{inputenc}
\usepackage{hyperref}

\hypersetup{
    pdfnewwindow=true,      
    colorlinks=true,       
    linkcolor=cyan,          
    citecolor=blue,        
    filecolor=magenta,      
    urlcolor=cyan           
}

\usepackage{graphicx}
\usepackage[all]{xy}
\usepackage{enumerate}
\usepackage[normalem]{ulem}
\usepackage{bm}
\usepackage{tikz}
\usetikzlibrary{fadings}

\numberwithin{equation}{section}

\newcommand{\one}{\ensuremath{\mathbf{1}}}

\newcommand{\moa}{\ensuremath{\operatorname{\mathsf{MOA}}}}
\newcommand{\imoa}{\ensuremath{\operatorname{\mathsf{IMA}}}}
\newcommand{\fmoa}{\ensuremath{\operatorname{\mathsf{FMA}}}}
\newcommand{\fmoap}{\ensuremath{\operatorname{\mathsf{FMA}_{\text{p}}}}}
\newcommand{\maxid}{\ensuremath{\operatorname{\mathsf{MMA}}}}
\newcommand{\mmoa}{\ensuremath{\operatorname{\mathsf{MMA}}}}
\newcommand{\gmoa}{\ensuremath{\operatorname{\mathsf{GMA}}}}
\newcommand{\ima}{\ensuremath{\operatorname{\mathsf{IMA}}}}
\newcommand{\fma}{\ensuremath{\operatorname{\mathsf{FMA}}}}
\newcommand{\mma}{\ensuremath{\operatorname{\mathsf{MMA}}}}
\newcommand{\gma}{\ensuremath{\operatorname{\mathsf{GMA}}}}
\newcommand{\dia}{\ensuremath{\operatorname{\mathsf{DMA}}}}
\newcommand{\Si}{\ensuremath{\operatorname{\mathbf{Si}}}}

\newcommand{\ulor}{\ensuremath{\;\underline{\lor}\;}}

\newcommand{\B}{\ensuremath{\mathfrak{B}}}
\newcommand{\A}{\ensuremath{\mathfrak{A}}}
\newcommand{\K}{\ensuremath{\mathbf{K}}}

\newcommand{\V}{\ensuremath{\mathbf{V}}}
\newcommand{\F}{\ensuremath{\mathfrak{F}}}
\renewcommand{\L}{\ensuremath{\mathcal{L}}}
\newcommand{\Q}{\ensuremath{\mathbf{Q}}}

\renewcommand{\phi}{\varphi}
\renewcommand{\epsilon}{\varepsilon}

\newcommand{\restrict}{\mathop{\upharpoonright}}

\DeclareMathOperator{\Eq}{\mathbf{Var}}
\DeclareMathOperator{\QEq}{\mathbf{QVar}}
\DeclareMathOperator{\ult}{Ult}
\DeclareMathOperator{\Ult}{Ult}
\DeclareMathOperator{\At}{At}

\newcommand{\Log}{\ensuremath{L}}

\newcommand{\Bc}{\ensuremath{B^{\text{c}}}}
\newcommand{\Bo}{\ensuremath{B^{\text{o}}}}

\newcommand{\Cm}{\ensuremath{\operatorname{{\mathsf{Cm}}}}}   

\newcommand{\Lan}{\ensuremath{\mathsf{Lan}}}
\newcommand{\df}{\ensuremath{\overset{\mathrm{df}}{=}}}
\newcommand{\aright}{``$\Implies$'': \ }
\newcommand{\aleft}{``$\Leftarrow$'': \ }

\newcommand{\klam}[1]{\ensuremath{\langle #1 \rangle}}
\newcommand{\set}[1]{\ensuremath{\{#1\}}}
\newcommand{\z}{\emptyset}

\newcommand{\tand}{\text{ and }}
\newcommand{\tor}{\text{ or }}
\newcommand{\tiff}{if and only if \ }

\newcommand{\Iff}{\Longleftrightarrow}
\newcommand{\Iffdf}{\overset{\mathrm{df}}{\Longleftrightarrow}}
\newcommand{\Implies}{\ensuremath{\Rightarrow}}
\renewcommand{\implies}{\rightarrow}

\newcommand{\conv}[1]{#1\ensuremath{~\breve{}~}}

\newcommand{\poss}[1]{\ensuremath{\klam{ #1 }}}
\newcommand{\sbe}{``$\subseteq$'': \ }
\newcommand{\spe}{``$\supseteq$'': \ }
\newcommand{\wlg}{w.l.o.g.\ }
\newcommand{\card}[1]{\vert #1 \vert}

\newcommand{\Var}{\ensuremath{\mathsf{Var}}}
\newcommand{\Fml}{\ensuremath{\mathsf{Fml}}}
\newcommand{\Model}{\ensuremath{\mathfrak M}}

 \newcommand{\sub}{\mathbf{S}}
  \renewcommand{\hom}{\mathbf{H}}
 \newcommand{\prodsu}{\mathbf{P_u}}

\newcommand{\dom}{\operatorname{dom}}

\newcommand{\da}[1]{\ensuremath{\mathop{\downarrow}#1}}  
\newcommand{\ua}[1]{\ensuremath{\mathop{\uparrow}#1}}  

\renewcommand{\1}{\ensuremath{\mathbf{1}}}
\renewcommand{\2}{\ensuremath{\mathbf{2}}}

\newcommand{\onto}{\twoheadrightarrow}

\theoremstyle{plain}

\newtheorem{theorem}{Theorem}[section]
\newtheorem{lemma}[theorem]{Lemma}
\newtheorem{corollary}[theorem]{Corollary}

\theoremstyle{definition}
\newtheorem{example}[theorem]{Example}

\numberwithin{equation}{section}

\parskip=\medskipamount
\title{Closure algebras of depth two with extremal relations:
Their frames, logics, and structural completeness}

\author{Ivo D\"untsch \& Wojciech Dzik}

\address{Ivo D\"untsch\\
Computer Science Department\\
Brock University\\	
St Catharines, Ontario\\
Canada}

\email{\href{mailto:duentsch@brocku.ca}{duentsch@brocku.ca}}

\address{Wojciech Dzik \\
Institute of Mathematics \\
University of Silesia \\
Katowice, Poland,
}
\email{\href{mailto:wojciech.dzik@us.edu.pl}{wojciech.dzik@us.edu.pl}}

\begin{document}
\maketitle

\hfill{\emph{Dedicated to our dear colleague Ivo G. Rosenberg}}

\begin{abstract}
\noindent We consider varieties generated by finite closure algebras whose canonical relations have two levels, and whose restriction to a level is an ``extremal'' relation, i.e. the identity or the universal relation. The corresponding logics have frames of depth two, in which a level consists of a set of simple clusters or of one cluster with one or more elements.
\end{abstract}


\section{Introduction}\label{sec:intro}

In~\cite{dd21} we have investigated \emph{ideal algebras} $\klam{B,f}$ which are closure algebras in which the set $B^{\text{c}}$ of closed elements is an ideal with the top element added. Our starting point was the unary discriminator $f^\one$, for which
\begin{gather}\label{fone}
f^\one(x) \df
\begin{cases}
0, &\text{if } x = 0, \\
1, &\text{otherwise}.
\end{cases}
\end{gather}
$f^\one$ is the largest element in the additive semilattice of modal operators on $B$ which was investigated in~\cite{ddo19}, and the set \Bc\ of its closed elements is $\set{0,1}$; in other words, $\klam{B,f^\one}$ is an ideal algebra with associated ideal $\set{0}$. The other extreme is the closure algebra in which every element is closed, i.e. where $f$ is the identity, and $B^{\text{c}} = I = B$. It turned out that the equational class generated by all ideal algebras is a locally finite positive universal class, and that the canonical  frame of a non-simple finite ideal algebra has depth two and consists of a set of simple clusters on the first level and a single cluster on the second level; in other words, the canonical frame relation restricted to the lower level is the identity and its restriction to the upper level is the universal relation. 

In the present paper we investigate  closure algebras related to ideal algebras and their logics. Our first case are closure algebras in which the set of closed elements is a filter $F$ with the smallest element added, followed by two other classes of algebras for which the set of closed elements may also be  related to ideals or filters. All these have the common property that their depth is at most two, and that the canonical relation of a finite algebra consists of at most two levels in which the restriction of the frame relation to a level is either universal or the identity.

The structure of the paper is as follows: After the introductory section, we briefly recall some facts about ideal algebras and their logic, followed by a section on filter algebras and their logic. Then we consider the remaining cases for algebras and frames of depth two with extremal relations and their logics. These sections are followed by a section in which we exhibit meet and join  of the varieties generated by the classes of algebras we have considered.  We close with some remarks on quasivarieties of algebras of depth two and structural completeness.

\section{Notation and First Definitions}

A \emph{frame} is a structure $\F = \klam{W,R}$, where $W$ is a nonempty set, and $\mathrel{R}$ is a binary relation on $W$. The identity relation on $W$ is denoted by $1'$. For $x \in W$ we set $R(x) \df \set{y: x\mathrel{R}y}$. The \emph{converse of $R$} is the relation $\conv{R} \df \set{\klam{y,x}: x\mathrel{R}y}$. $R$ is called \emph{convergent} if $x\mathrel{R}y$ and $x\mathrel{R}z$ imply the existence of some $w \in W$ such that $y\mathrel{R}w$ and $z\mathrel{R}w$. $R$ is called \emph{directed}, if for all $x,y \in W$ there is some $z \in W$ such that $x\mathrel{R}z$ and $y\mathrel{R}z$.
A reflexive and transitive relation is called a \emph{quasiorder}. An antisymmetric quasiorder is called a \emph{partial order}. A partial order can be obtained from a quasiorder $R$ on the classes of the equivalence relation defined by $x\mathrel{\theta_R} y \Iffdf x\mathrel{R}y \tand y\mathrel{R}x$, sometimes called the \emph{$T_0$ quotient of $R$}. Following~\cite[p. 75f]{seg71}, we call the classes of $\theta_R$ \emph{clusters}. A cluster is \emph{simple}, if $\card{\theta_R(x)} = 1$, otherwise it is called \emph{proper}. The relation $\leq_R$  defined by $\theta_R(x) \leq_R \theta_R(y) \Iffdf x\mathrel{R}y$ is a partial order. If $W$ is finite, then $\leq_R$ has minimal and maximal elements, i.e. a smallest and a largest level. A frame is called a \emph{frame of depth $n, 1 \leq n$}, if the length of any maximal chain of $\leq_R$ is at most $n$, and this is attained for some maximal chain.

An algebra is called \emph{trivial}, if it has exactly one element. If $\A$ and $\B$ are algebras of the same type, we write $\A \mathrel{\leq} \B$, if $\A$ is a subalgebra of $\B$. $\mathbf{Sub}(\B)$ denotes the collection of subalgebras of $\B$.

As no generality is lost, we shall tacitly assume that a class of algebras is closed under isomorphic copies. If no confusion can arise, we will often refer to an algebra simply by its base set. If $\mathbf K$ is a class of algebras of the same similarity type, we denote by $\mathbf{H(K)}$ the collection of all homomorphic images of $\mathbf K$, by $\mathbf{S(K)}$ the collection of all subalgebras of $\mathbf K$, by $\mathbf{P(K)}$ the collection of all products of elements of $\mathbf K$, by $\mathbf{Pu}(\K)$ the class of all ultraproducts of members of \K, and by $\mathbf{Si}(\K)$ the class of all subdirectly irreducible members of \K.  The equational class (variety) $\mathbf{HSP(K)}$ generated by $\mathbf K$ is denoted by $\Eq(\mathbf K)$. It is well known that an equational class is generated by its finitely generated members. A class of algebras is called \emph{locally finite} if every finitely generated member is finite. Thus, a locally finite variety is generated by its finite members. A variety $\V$ is called \emph{tabular}, if it is generated by a finite algebra, and \emph{pretabular} if it is not tabular, but every proper subvariety is tabular~\cite[p. 104]{Blok1980-2}.

A \emph{quasivariety} is a class $\mathbf{K}$ of algebras which contains a trivial algebra, and for which $\mathbf{SPPu}(\K) = \K$. It is well known that $\K$ is a quasivariety \tiff it can be axiomatized by quasiidentities, i.e. by sentences of the form $(\tau_1 = \sigma_1 \; \& \ldots \& \; \tau_n = \sigma_n) \Rightarrow \tau = \sigma$, see~\cite[Theorem 2.25]{bs_ua}. The quasivariety generated by $\K$ is denoted by $\QEq(\K)$.

Throughout, $\B = \klam{B,+,\cdot,-,0,1}$  is a Boolean algebra (BA) with at least two elements. We shall usually identify a Boolean algebra with its underlying set, so that we may write $B$ instead of \B. If $M \subseteq B$, then $M^+ \df \set{x \in M: x \neq 0}$, and $-M \df \set{-x:x \in M}$. If $a \in B$, then $\da{a} \df \set{x: x \leq a}$ is the ideal generated by $a$, and $\ua{a} \df \set{x: x \geq a}$ is the filter generated by $a$. Moreover, $\At(B)$ denotes the set of atoms of $B$, and $\Ult(B)$ is the set of ultrafilters of $B$.

A \emph{modal operator} on $B$ is a mapping $f:B \to B$ which satisfies, for all $a,b \in B$,
\begin{xalignat*}{2}
f(0) &= 0, &&\text{normality}, \\
f(a+b) &= f(a) + f(b), &&\text{additivity.}
\end{xalignat*}
If $f$ is a modal operator, $\klam{B,f}$ is called a \emph{modal algebra}. Modal algebras were investigated by J{\'o}nsson and Tarski~\cite{jt51} under the name of \emph{Boolean algebras with operators}, where many of their properties can be found. The \emph{dual of a modal operator $f$} is the function $f^\partial: B \to B$, defined by $f^\partial(a) = -f(-a)$. An ideal $I$ of $B$ is \emph{closed}, if $f(a) \in I$ for every $a \in I$, and a filter $F$ is called \emph{open}, if $f^\partial(a) \in F$ for all $a \in F$. We denote the variety of modal algebras by \moa, and the lattice of its subvarieties by $\Lambda(\moa)$; join and meet in $\Lambda(\moa)$ are denoted by $\lor$ and by $\land$, respectively. The class of algebras of the form $\klam{B,f^\one}$ is denoted by \dia, see \eqref{fone}. It is well known that \dia\ is the class of simple monadic algebras, see e.g. \cite{monk70}.

We shall frequently make use of an instance of J{\'o}nsson's Lemma:
\begin{lemma}\label{lem:jon} \cite[Corollary 3.2]{jon67}
If $\K$ is a class of modal algebras, then $\Si\Eq(\K) \subseteq \hom\sub\prodsu(\K)$.
\end{lemma}
In particular, if $\K$ is axiomatizable by first order positive universal sentences, then $\Si\Eq(\K) \subseteq \K$, see e.g.~ \cite[Chapter 5]{ck71}.

If $\B =\klam{B,f}$ is a modal algebra, then the structure $\klam{\ult(B), R_f}$ with
\begin{gather*}
u\mathrel{R_f} v \Iff f[v] \subseteq u
\end{gather*}
is called the \emph{canonical frame} or \emph{ultrafilter extension} of \B. For finite algebras this has a simple form: If $a,b \in \At(B)$, and $F_a, F_b$ are the principal ultrafilters generated by $a$, respectively, by $b$, then
\begin{gather}\label{RfAt}
F_a \mathrel{R_f} F_b \Iff \set{f(p): b \leq p} \subseteq F_a \Iff a \leq f(b).
\end{gather}
If $\F = \klam{W,R}$ is a frame, then the structure $\Cm(\F) \df \klam{2^W, \poss{R}}$ is called the \emph{complex algebra of $\F$}. Here, for $X \subseteq W$,
\begin{gather*}
\poss{R}(X) \df \set{x \in W: R(x) \cap X \neq \z}.
\end{gather*}
\begin{theorem}\label{thm:framealg} \cite[Theorem 3.9]{jt51}
\begin{enumerate}
\item $\Cm(\F)$ is  complete and atomic, and $\poss{R}$  is a completely additive normal operator.
\item If $\B = \klam{B,f}$ is a complete atomic Boolean algebra and $f$ is completely additive, then $\B$ is isomorphic to some $\Cm(\klam{W,R})$.
\end{enumerate}
\end{theorem}

Two modal operators $f,g$ on $B$ are called \emph{conjugate}, if, for all $a,b \in B$,
\begin{gather}\label{def:conj}
f(a) \cdot b = 0 \Iff g(b) \cdot a = 0.
\end{gather}

\begin{theorem}\label{thm:conj}  \cite[Theorem 3.6.]{jt51}
If $f,g$ are conjugate modal operators on $B$, then $f$ and $g$ are completely additive, and $R_f$ is the converse of $R_g$. Conversely, if $R$ is a binary relation on $W$, then $\poss{R}$ and $\poss{\conv{R}}$ are conjugate.
\end{theorem}

A modal operator $f$ is called a \emph{closure operator}~\cite{mkt44}, if it satisfies
\begin{enumerate}[{Cl}$_1$]
\item $a \leq f(a)$,
\item $f(f(a)) = f(a)$.
\end{enumerate}
In this case, $\klam{B,f}$ is called a \emph{closure algebra}; observe that the complex algebra of a quasiordered frame is a closure algebra. Closure algebras have equationally definable congruences~\cite[Example 18, p. 204f]{bp82}. It is well known that the congruences of a closure algebra correspond to the closed ideals of $B$.  The dual $\klam{B,g}$ of a closure algebra is called an \emph{interior algebra}, and $g$ an interior operator. Thus, $g(1) = 1$, $g$ is multiplicative, and $g$ satisfies
\begin{enumerate}[{Int}$_1$]
\item $g(a) \leq a$,
\item $g(g(a)) = g(a)$.
\end{enumerate}
The classes of closure algebras and interior algebras are definitionally  equivalent.  If $\klam{B,f}$ is a closure algebra we denote its corresponding interior algebra $\klam{B, f^\partial}$ by $B^\partial$. An element $x$ of a closure algebra $\klam{B,f}$ is called \emph{closed} if $f(x) = x$, and \emph{open}, if $f^\partial(x) = x$. The trivial algebra is denoted by \1, and the two element closure algebra is denoted by \2.


 For later use we mention
\begin{lemma}\label{lem:si} \cite{Blok76,blok80a} 
\begin{enumerate}
\item A closure algebra is subdirectly irreducible \tiff it has a smallest nontrivial closed ideal.
\item If~$\V_1$ and $\V_2$ are varieties of closure algebras, then $\Si(\V_1 \lor \V_2) = \Si(\V_1) \cup \Si(\V_2)$.
\end{enumerate}
\end{lemma}

The following result will be useful in the sequel:

\begin{theorem}\label{thm:bdcl} \cite[p. 190f]{bd75}
Suppose that $D$ is a bounded sublattice of $B$. Then, $D$ is the set of closed elements of a closure operator $f$ on $B$ \tiff  $\ua{b} \cap D$ has a smallest element for each $b \in B$. In this case, $f(b) = \min\;(\ua{b} \cap D)$.
\end{theorem}

The following property is decisive for the classes of algebras which we consider: A closure algebra $\klam{B,f}$ is said to have \emph{depth two}, if it satisfies

\begin{enumerate}[{B}$_1$]
\setcounter{enumi}{1}
\item $f(f^\partial(x) \cdot f(f^\partial(y)) \cdot -y) \leq x.$ \label{b2}
\end{enumerate}
The name comes from the fact that the canonical relation of a finite closure algebra satisfying $\mathrm{B}_{\ref{b2}}$ has depth two, see Table~\ref{tab:logcond}.

For unexplained notation we refer the reader to~\cite{bs_ua} for universal algebra, and to~\cite{kop89} for Boolean algebras. A thorough investigation of varieties of interior algebras is~\cite{Blok76}.

\section{Modal logics and their semantics}\label{sec:log}

Modal logic extends classical propositional logic by a unary logical connective $\Diamond$ representing ``possibility''. Formulas and terms are defined recursively in the usual way. The \emph{dual of $\Diamond$} is the operator $\Box$ (``necessity'') defined by $\Box \phi = \neg\Diamond(\neg\phi)$. A \emph{modal logic} $\Log$  is a set of modal formulas that contains all propositional tautologies and is closed under modus ponens and substitution. $\Log$ is \emph{normal}, if it contains the formula

\begin{enumerate}[\texttt{K}]
\item $\Box(p \implies q) \implies (\Box p \implies \Box q)$.
\end{enumerate}
and is closed under necessitation, i.e. $\phi \in L$ implies $\Box\phi \in L$. Equivalently, $\Log$ is normal if it contains the formulas $\Diamond \bot \leftrightarrow \bot$ and $\Diamond (\phi \lor \psi) \leftrightarrow \Diamond \phi \lor \Diamond \psi$. In the sequel a logic is assumed to be normal with a set $\Var$ of variables and a set \Fml\ of formulas in a suitable language \Lan. If $\Gamma \subseteq \Fml$, then the smallest normal logic containing $\Gamma$ is denoted by $\Log_\Gamma$. If $\Log \subseteq \Log'$, then $\Log'$ is called an \emph{extension of $\Log$}.  It is well known that the class of normal modal logics forms a lattice under $\subseteq$, dually isomorphic to the lattice of varieties of modal algebras, see e.g.~\cite{blok80a}.

A class $\mathsf M$ of modal algebras \emph{validates} or \emph{determines a logic $\Log$} if $\B \models \phi$ for all theorems $\phi$ of $\Log$ and all $\B \in \mathsf M$.  With some abuse of notation, we let $\Eq(\Log)$ be the class of modal algebras that satisfy all theorems of $\Log$. In the other direction, the set $\Delta_{\mathsf M}$ of all formulas valid in all members of $\mathsf M$ is a normal logic $\Log_{\mathsf M}$~\cite[p. 178]{gol89}. If $\mathsf{M} = \set{\B}$, then $\Log_{\text{\textsf{M}}} =  \Log_{\Eq(\set{\B})}$.

Next, we turn to frame semantics. Validity in a frame is defined as follows: A \emph{model $\Model$} (for $\Fml$) is a structure $\klam{W,R,v}$ where $\klam{W,R}$ is a frame, and $v: \Var \to 2^W$ is a valuation function which is extended over \Fml\ as follows:
\begin{align*}
v(\neg\phi) &= W \setminus v(\phi) = \set{w \in W: w \not\in v(\phi)}, \\
v(\phi \land \psi) &= v(\phi) \cap v(\psi),\\
v(\top)&= W, \\
v(\Diamond(\phi)) &= \set{w: (\exists u)[u \in v(\phi) \tand u\mathrel{R}w}. 
\end{align*}

A formula $\phi$ is \emph{true in the model $\Model = \klam{W,R,v}$}, written as $\Model \models_v \phi$, if $v(\phi) = W$. $\phi$ is \emph{true in the frame $\F = \klam{W,R}$}, written as $\F \models \phi$,  if $\klam{W,R,v} \models \phi$ for all valuations $v: \Fml \to 2^W$. If $K$ is a class of frames, we write $K \models \phi$ just in case $\F \models \phi$ for all $\F \in K$.
Furthermore, if $\Gamma$ is a set of formulas, then we write $\F \models \Gamma$ if $\F \models \phi$ for all $\phi \in \Gamma$, and we write $K \models \Gamma$, if $\F \models \Gamma$ for all $\F \in K$.

A frame \F\ \emph{determines the logic} $\Log_{\F} = \set{\phi: \F \models \phi}$, and a class $K$ of frames determines the logic $\Log_K = \set{\phi: K \models \phi} = \bigcap \set{\Log_{\F}: \F \in K}$. Conversely, if $\Gamma$ is a set of formulas, then $K_\Gamma \df \set{\F: \F \models \Gamma}$ is the class of frames that determine $\Gamma$; note that $K_\Gamma = K_{\Log_\Gamma}$. A class $K'$ of frames is called \emph{modally definable} if it is of the form $K_{\Log}$ for some  logic $\Log$, i.e. of the form $K_{\Log_{K'}}$.

A logic $\Log'$ is \emph{Kripke (or frame) complete}, if it is determined by a class $K'$ of frames, i.e. if $\Log' = \Log_{K'}$. This is true \tiff $\Log' = \Log_{K_{\Log'}}$. Although not all logics are Kripke complete, see e.g.~\cite{bs84} Section 19, and~\cite{blok80a},  in this paper we will deal only with Kripke complete logics.  Table \ref{tab:logcond} shows frame conditions for various extensions of $\K$. There, the axioms are named in \texttt{teletype}, and the corresponding logics generated by them are denoted in \textbf{boldface}.
\begin{table}[tb]
\caption{Modal logic axioms}\label{tab:logcond}
\centering
{\small
\begin{tabular}{llll}
Name & Axiom & Frame condition \\ \hline
\texttt{D} & $\Box p  \implies  \Diamond p $ & $R$ is serial, i.e. $\dom(R) = W$. \\
\texttt{T} & $p \implies  \Diamond p $ & $R$ is reflexive. \\
\texttt{4} & $\Diamond\Diamond p \implies  \Diamond p $ & $R$ is transitive.\\
\texttt{B} & $p \implies  \Box\Diamond p $ & $R$ symmetric\\
\texttt{B}$_2$ & $\Diamond(\Box q \land \Diamond\Box p \land \neg p) \implies q$  & $(x\mathrel{R}y \tand y\mathrel{R}z) \Implies (y\mathrel{R}x \tor z\mathrel{R}y)$. \\
\texttt{Dum} & $ \Box(\Box(p \implies \Box p) \implies p ) \land \Diamond\Box p \; \implies p$ & $R$ is a quasiorder in which all but the final cluster is simple.  \\
\texttt{Grz} & $\Box(\Diamond(p \land  \Diamond(-p))\lor p) \implies  p$ &There is no infinite strictly ascending $R$-chain. \\
\texttt{M} or \texttt{.1} & $\Box\Diamond p \implies  \Diamond\Box p $  & \emph{None - the logic \textbf{KM} is not Kripke complete.} \\%
\texttt{G} or \texttt{.2} & $\Diamond\Box p  \implies  \Box\Diamond p$ & $R$ is convergent \\
\texttt{H} or \texttt{.3}
& $\Box(\Box p \implies q) \lor \Box(\Box q \implies p)$  & $(x\mathrel{R}y \tand x\mathrel{R}z) \Implies (y\mathrel{R}z \tor z \mathrel{R} y)$
\\
\texttt{R1} & $p  \land  \Diamond\Box p \implies  \Box p $ & $(x\mathrel{R}y \tand x \neq y \tand x\mathrel{R}z) \Implies z\mathrel{R}y$
\end{tabular}
}
\end{table}
Since the nomenclature used in the literature is not uniform, we list some alternative denotation:
\begin{center}
$\mathbf{S4} \df \mathbf{KT4}$, \
\textbf{S4.1} \df\ \textbf{S4M},\
\textbf{S4.2} \df\ \textbf{S4G}, \
\textbf{S4.3} \df\ \textbf{S4H}, \
\textbf{S4.4} \df\ \textbf{S4R1}, \
\textbf{S5} \df\ \textbf{S4B}.
\end{center}

We note that the \textbf{T4} extension of \textbf{KM}, i.e. \textbf{S4M}, is Kripke complete: If $R$ is a quasiorder, then the condition that verifies \textbf{M} is $(\forall x)(\exists y)[x\mathrel{R}y \tand  (\forall z)(y\mathrel{R}z \Implies y = z)]$.  As the logics whose canonical frames have depth two are of particular importance in the sequel, we shall mention the following:

\begin{theorem}\label{lem:423}
$\mathbf{S4.2B_2} = \mathbf{S4.3B_2}$.
\end{theorem}
\begin{proof}
Since both logics are Kripke complete, it is enough to show that for a quasiorder $R$ of depth two, \texttt{.2} holds \tiff \texttt{.3} holds.

\aright Suppose that $R$ is convergent, and let $x\mathrel{R}y$ and $x\mathrel{R}z$; we need to show that $y \mathrel{R} z$ or $z \mathrel{R} y$. If $y \mathrel{R} x$, then $y \mathrel{R} z$, since $R$ is transitive; similarly, $z \mathrel{R} x$ implies $z \mathrel{R} y$. Suppose that $y \mathrel{(-R)} x$ and $z \mathrel{(-R)} x$. Since $R$ is convergent, there is some $w$ such that $y \mathrel{R} w$ and $z \mathrel{R} w$; furthermore, \texttt{B}$_2$ and our assumption imply $w \mathrel{R} y$ and $w \mathrel{R} z$. Since $y \mathrel{R} w$, and by the transitivity of $R$, we obtain $y \mathrel{R} z$, and, similarly, $z \mathrel{R} y$.

\aleft Suppose that $R$ satisfies \texttt{.3}, and let $x\mathrel{R}y$ and $x\mathrel{R}z$. By \texttt{.3}, we can \wlg assume that $y\mathrel{R}z$ otherwise we exchange $y$ and $z$. Since $R$ is reflexive, we also have $z\mathrel{R}z$, and we set $w \df z$. Observe that \texttt{B}$_2$ was not needed for this direction.
\end{proof}

\section{Ideal algebras}

Our starting point are modal algebras $\klam{B,f}$ for which there is an ideal $I$ of $B$ such that
\begin{gather*}
f(x) =
\begin{cases}
x, &\text{if } x \in I, \\
1, &\text{otherwise.}
\end{cases}
\end{gather*}
Algebras of this form are called \emph{ideal algebras}; the class of all ideal algebras is denoted by \imoa. Ideal algebras generalize the discriminator algebra $\klam{B, f^\one}$: If $I = \set{0}$, then $f = f^\one$. In~\cite{dd21} we have investigated the class \imoa, and its main properties are as follows:
\begin{theorem}\label{thm:idalg} 
\begin{enumerate}
\item \imoa\ is a locally finite positive universal class, axiomatized by
\begin{gather}
(\forall x)[f(x) = x \tor f(x) = 1].
\end{gather}
\item $\klam{B,f}$ is an ideal algebra \tiff the set $B^{\text{c}}$ of closed elements has the form $I \cup \set{1}$, where $I$ is an ideal of $B$.
\item An ideal algebra is a closure algebra of depth at most two.
\item If $\klam{B,f} \in \Eq(\imoa)$ is subdirectly irreducible, then it is an ideal algebra. An ideal algebra $\klam{B,f}$ is subdirectly irreducible \tiff its defining ideal $I$ has the form $I = \set{0}$, or it is generated by an atom of $B$.
\end{enumerate}
\end{theorem}
Our initial concern was algebraic, and it was something of a surprise to us that ideal algebras constitute the algebraic semantics of Soboci{\'n}ski's logic \textbf{S4.4}~\cite{sob64}, also known as $\mathbf{S4.3DumB_2}$.
\begin{theorem}
\begin{enumerate}
\item $\Eq(\imoa)$ is generated by the complex algebras of finite frames $\klam{W,R}$ where $R$ is a quasiorder with two levels, such that the lower level consists of simple clusters, and the upper level consists of one cluster, see Figure \ref{fig:cs_idalg}.
\item $\Eq(\imoa) = \Eq(\mathbf{S4.3DumB_2})$.
\end{enumerate}
\end{theorem}
The canonical relation of an ideal algebra depicted in Figure \ref{fig:cs_idalg} is called an \emph{iu--relation}, and its corresponding operator an \emph{iu--operator}, denoted by $f^{\text{iu}}$. This notation indicates that the restriction of the canonical relation $R_f$ to the lower level is the \underline{\bf i}dentity, and the restriction to the upper level is the \underline{\bf u}niversal relation. This observation motivated us to consider the corresponding situations of {iu,uu,iu} relations and operators.
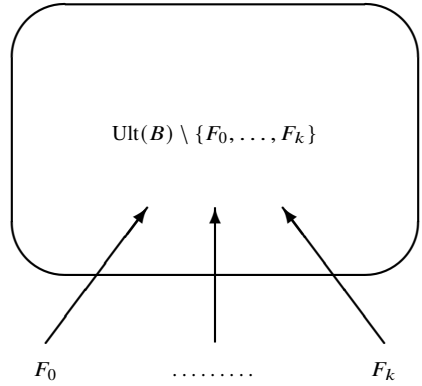
\begin{figure}[htb]
\caption{The canonical relation of a finite ideal algebra of depth two}\label{fig:cs_idalg}
\vspace{5mm}
\centering
\setlength{\unitlength}{1776sp}%
\begingroup\makeatletter\ifx\SetFigFont\undefined%
\gdef\SetFigFont#1#2#3#4#5{%
  \reset@font\fontsize{#1}{#2pt}%
  \fontfamily{#3}\fontseries{#4}\fontshape{#5}%
  \selectfont}%
\fi\endgroup%
\begin{picture}(5732,5331)(1859,-5409)
\thicklines
{\color[rgb]{0,0,0}\put(2363,-4830){\vector( 3, 4){1417.200}}
}%
{\color[rgb]{0,0,0}\put(4725,-4830){\vector( 0, 1){1889}}
}%
{\color[rgb]{0,0,0}\put(7088,-4830){\vector(-3, 4){1417.200}}
}%
\put(4725,-1996){\makebox(0,0)[b]{\smash{{\SetFigFont{8}{9.6}{\rmdefault}{\mddefault}{\updefault}{\color[rgb]{0,0,0}$\Ult(B) \setminus \{F_{0}, \ldots, F_{k}\}$}%
}}}}
\put(2363,-5303){\makebox(0,0)[b]{\smash{{\SetFigFont{8}{9.6}{\rmdefault}{\mddefault}{\updefault}{\color[rgb]{0,0,0}$F_{0}$}%
}}}}
\put(4725,-5303){\makebox(0,0)[b]{\smash{{\SetFigFont{8}{9.6}{\rmdefault}{\mddefault}{\updefault}{\color[rgb]{0,0,0}\ldots \ldots \ldots}%
}}}}
\put(7088,-5303){\makebox(0,0)[b]{\smash{{\SetFigFont{8}{9.6}{\rmdefault}{\mddefault}{\updefault}{\color[rgb]{0,0,0}$F_{k}$}%
}}}}
{\color[rgb]{0,0,0}\put(4725,-1996){\oval(5672,3780)}
}%
\end{picture}%
\end{figure}
%

\section{Filter algebras}\label{sec:fmoa}

Ideal algebras are characterized by the property that the set $B^{\text{c}}$ of closed elements has the form $B^{\text{c}} = I \cup \set{1}$ for some ideal of $B$. Similarly, we call a closure algebra $\klam{B,f}$ a \emph{filter algebra}, if $B \in \set{\1,\2}$ or $B^{\text{c}}$ has the form $\set{0} \cup F$ for some filter $F$ of $B$ which is not equal to $\set{1}$. In the sequel we shall use $F$ for the determining filter. $\klam{B,f}$ is called \emph{proper}, if $F \neq B$.  We denote the class of all filter algebras by \fmoa, and the class of proper filter algebras by \fmoap.  Note that apart from the trivial algebra, $\2$ is the only discriminator algebra in \fmoa.

Even though for any proper ideal $I$, $I \cup \set{1}$ can be the set of closed elements, this is not the case for filter algebras:
\begin{theorem}
Let $\klam{B,f}$ be a filter algebra with respect to $F$.  Then, $F$ is principal.
\end{theorem}
\begin{proof}
If $F = B$, then $F = \ua{0}$. Suppose that $F \neq B$; then, $\set{0} \cup F$ is a proper $0,1$-sublattice of $B$. We already know from Theorem~\ref{thm:bdcl} that a $0,1$-sublattice $D$ of $B$ is the set of closed elements \tiff $\ua{b} \cap D$ has a smallest element for all $b \in B$; in this case, $f(b) = \min\; (\ua{b} \cap D)$. Choose some $x \in F \setminus \set{1}$;  then, $-x \neq 0$ and $f(-x) \in F$, since $f(-x)$ is closed. It follows that $x \cdot f(-x) \in F$.

Assume that $\ua{x} \cdot f(-x) \subsetneq F$. Then, there is some $y \in F$ such that $y \lneq x \cdot f(-x)$. If $-x + y = f(-x)$, then $y = x \cdot (-x + y) =  x \cdot f(-x)$, a contradiction. Hence, $-x + y \lneq f(-x)$, and now $-x + y \in F$ shows that $f(-x)$ is not the smallest element of $F$ above $-x$, a contradiction.
\end{proof}
\begin{theorem}\label{thm:fa}
If $\klam{B,f}$ is a filter algebra with determining filter $\ua{a}$, then for all $x \in B$,
\begin{gather}\label{fmoa:princ}
f(x) =
\begin{cases}
0, &\text{if } x = 0,\\
a+x, &\text{otherwise}.
\end{cases}
\end{gather}
\end{theorem}
\begin{proof}
Let $x \neq 0$. Then, $x \leq f(x)$, and $a+x$ is the smallest element of $\set{0} \cup \ua{a}$ above $x$.
\end{proof}

\begin{theorem}\label{thm:fmasi}
A filter algebra $\klam{B,f}$ is subdirectly irreducible \tiff $\card{B} = 2$ or it is proper.
\end{theorem}
\begin{proof}
\aright Suppose that $\klam{B,f}$ is subdirectly irreducible. If $\klam{B,f}$ is not proper, then $F = B$, and each element is closed.  Thus, $\card{B} = 2$.

\aleft If $\card{B} = 2$, then it is simple, hence, subdirectly irreducible. Otherwise, suppose that $\klam{B,f} \in \fma$ with respect to some $a \neq 0$. Then, $\da{a}$ is a nontrivial closed ideal, since $f(x) = a$ \tiff $0 \lneq x \leq a$. Since every nontrivial closed ideal contains $a$, $\da{a}$ is the smallest nontrivial closed ideal, hence, $B$ is subdirectly irreducible.
\end{proof}

\begin{theorem}\label{thm:fmopu}
$\fmoa$ is a positive universal class.
\end{theorem}
\begin{proof}
Suppose that $\klam{B,f} \in \fma$ with respect to $a$. If $a = 0$, i.e. $f = 1'$, then, clearly, each subalgebra and homomorphic image of $\klam{B,f}$ is in $\fma$. Thus, suppose that $a \neq 0$, and let $\klam{A, g}$ be a subalgebra of $\klam{B,f}$, $\card{A} \gneq \2$. Let $F' \df \ua{a} \cap A$; then, $F'$ is a filter of $A$, and $A^{\text{c}} = \Bc \cap A = \set{0} \cup F'$. Since $\card{A} \geq 4$, there is some $x \in A$ such that $x \neq 0$ and $-x \neq 0$. Now, $g(x) \cdot g(-x) = f(x) \cdot f(-x) = (a+x) \cdot (a + -x) = a \in A$, and thus, $a \in F'$ and $a \neq 1$, since $\card{B} \geq \card{A} \gneq 2$.

Next, suppose that $I$ is a nontrivial congruence ideal of $\klam{B,f}$; then, $a \in I$.  Let $A \df B/I$ with $\pi: B \onto A$ the natural homomorphism;  \wlg we suppose that $A \neq \2$. Let $g: A \to A$ be the induced mapping defined by $g(\pi(x)) = \pi(f(x))$. This is well defined, since $I$ is a congruence ideal. We will show that $g$ is the identity, which implies that $\klam{A,g} \in \fma$: If $x = 0$, then $g(\pi(0)) = \pi(f(0)) = \pi(0)$. Otherwise, for any $y \in B$ we have
\begin{xalignat*}{2}
y \in g(\pi(x)) &\Iff y \in \pi(f(x)), &&\text{since $\pi$ is a homomorphism,} \\
&\Iff y \in \pi(a+x), &&\text{since $x \neq 0$,}\\
&\Iff (\exists z \in I)[y+z = x+a+z], \\
&\Iff (\exists z' \in I )[y + z' = x + z'], &&\text{set $z' \df a + z$, which is in $I$}, \\
&\Iff y \in \pi(x).
\end{xalignat*}
This completes the proof.
\end{proof}
\begin{corollary}\label{cor:sif}
\begin{enumerate}
\item Every subdirectly irreducible algebra in $\Eq(\fmoa)$ is a filter algebra.
\item $\Eq(\fmoa)$ is generated by $\fmoap$.
\end{enumerate}
\end{corollary}
\begin{proof}
1. Suppose that $\klam{B,f} \in \Eq(\fmoa)$ is subdirectly irreducible. By J{\'o}nsson's Lemma~\ref{lem:jon}, $\klam{B,f} \in \hom\sub\prodsu(\fmoa)$. Theorem~\ref{thm:fmopu} shows that \fmoa\ is closed under taking homomorphic images, subalgebras, and ultraproducts, hence, $\klam{B,f} \in \fmoa$.

2. $\Eq(\fmoa)$ is generated by its subdirectly irreducible algebras. From 1. above and Theorem \ref{thm:fmasi} we obtain that $\Eq(\fmoa)$ is generated by $\fmoap \cup \set{\2}$. Now observe that $\2$ is a homomorphic image of, for example, the four element proper filter algebra (which is unique up to isomorphism).
\end{proof}

Next we show that $\Eq(\fmoa)$ is locally finite; it will be more convenient to work with interior algebras as in~\cite{bd75}. Let $\klam{B,f}$ be a filter algebra with $f$ determined by $\ua{a}$; then the set $B^{\text{o}}$ of open elements is the set $\da{ -a} \cup \set{1}$: Clearly, $1$ is open. Let $x \neq 1$, i.e. $-x \neq 0$. Then,
\begin{gather*}
f^\partial(x) = x \Iff  -f(-x) = x \Iff f(-x) = -x \Iff a+ -x = -x \Iff a \leq -x \Iff x \leq -a.
\end{gather*}
Thus, $\klam{B,f}$ is a filter algebra \tiff $B^{\text{o}}$ has the form $\da b \cup \set{1}$ for some $b \in B$.

For each $n \in \omega^+$, we define $K_n$ to be (isomorphic to) the interior algebra generated by its atoms $a_1, \ldots, a_{n}$ with open elements $b_i: i \leq n$, where $b_0 \df 0$ and $b_i \df a_1 + \ldots + a_i$ for $0 \lneq i \leq n$~\cite[p. 32]{blok80a};  thus, $B^{\text{o}}$ is a chain of length $n+1$. The induced interior operator is obtained by $i(x) \df \max (\da{x} \cap \set{b_i: i \leq n})$.

For the class $\mathsf{I}$ of interior algebras we define $(\mathsf{I}:K_n) \df \set{A \in \mathsf{I}: K_n \not\in \mathbf{Sub}(A)}$. We will use the characterization of locally finite interior algebras of~\cite[4.3, p. 181]{Blok76}:
\begin{theorem}\label{thm:locfin}
Suppose that $\mathsf V$ is a variety of interior algebras. Then,  $\mathsf V$ is locally finite \tiff $\mathsf{V} \subseteq (\mathsf{I}:K_n)$ for some $n$.
\end{theorem}

\begin{theorem}\label{thm:flf}
$\Eq(\fmoa)$ is locally finite.
\end{theorem}
\begin{proof}
We will show that the class  $\mathsf K$ of algebras dual to those in $\fmoap$ is contained in $(\mathsf{I}:K_3)$.
Suppose that $\klam{B,i}$ is an interior algebra with open elements $B^{\text{o}} = \da{a} \cup \set{1}$, $a \neq 1$, i.e. $\klam{B,i}$ is the dual of a proper filter algebra. Assume that $K_3$ is embeddable into $B$. We may suppose that $K_3$ is a subalgebra of $B$ with atoms $b_1,b_2,b_3$ and open elements $K_3^{\text{o}} \df \set{0,b_1, b_1+b_2, 1}$; then, $K_3^{\text{o}} \subseteq \da{a}$. Since $B^{\text{o}}$ is an ideal and $b_1 + b_2 \in K_3^{\text{o}} \setminus \set{1} \subseteq \da a$, we have $b_2 \in \da a \cap K_3 \subseteq B^{\text{o}} \cap K_3$, hence, $b_2 \in K_3^{\text{o}}$, a contradiction. Thus, each algebra dual to a proper filter algebra is in $(\mathsf{I}:K_3)$. Since $(\mathsf{I}:K_3)$ is a variety, $\Eq(\mathsf K) \subseteq (\mathsf{I}:K_3)$. By Theorem~\ref{thm:locfin}, we obtain that $\Eq(\mathsf K)$ is locally finite, thus, so is $\Eq(\fmoa)$.
\end{proof}
The following is now immediate from Theorem \ref{thm:fmasi} and the previous theorem:
\begin{corollary}\label{cor:ui}
$\Eq(\fmoa)$ is generated by the finite members of $\fmoap$.
\end{corollary}
\begin{proof}
Each variety is generated by its finitely generated subdirectly irreducible members. By the previous theorem, these are finite, and by Corollary \ref{cor:sif} (2) we may choose these to be in \fmoap.
\end{proof}

Next, we shall look at the canonical frame of a proper finite filter algebra $\klam{B,f}$. Suppose that $f$ is  determined by the filter $\ua{a}$, $a \not\in \set{0,1}$, and that $B$ is generated as a Boolean algebra by its atoms $a_0, \ldots, a_n$. Let \wlg $a = a_0 + \ldots + a_m$ for some $m \lneq n$. Let $F,G \in \ult{(B)}$, and recall that the canonical relation $R \df R_{f}$ is defined by $F \mathrel{R} G$ if and only if $f[G] \subseteq F$, i.e. \tiff $\ua{a} \cap G \subseteq F$.
There are two cases:
\begin{enumerate}
\item $a \in F$: Then $\ua{a} \subseteq F$, and thus, $F\mathrel{R}G$ for all $G \in \ult(B)$. Observe that $a \in F$ \tiff $F = \ua{a_i}$ for some $i \leq m$.
\item $a \not\in F$: Then, $F = \mathop{\uparrow}a_i$ for some $m < i \leq n$. If $G = \mathop{\uparrow}a_j$, then $F \mathrel{R} G$ entails  $a + a_j \in F$, and subsequently, $a_j \in F$, since $a \not\in F$ and $F$ is an ultrafilter. It follows that $j = i$.
\end{enumerate}
The graph of $R$ is shown in Figure \ref{fig:filtalg}.
\begin{figure}[tb]
\caption{The canonical relation of a finite filter algebra of depth two}\label{fig:filtalg}
\vspace{5mm}
\centering
\setlength{\unitlength}{1558sp}%
\begingroup\makeatletter\ifx\SetFigFont\undefined%
\gdef\SetFigFont#1#2#3#4#5{%
  \reset@font\fontsize{#1}{#2pt}%
  \fontfamily{#3}\fontseries{#4}\fontshape{#5}%
  \selectfont}%
\fi\endgroup%
\begin{picture}(5702,4617)(1402,-6263)
{\color[rgb]{0,0,0}\thicklines
\put(4253,-4830){\oval(5672,2838)}
}%
{\color[rgb]{0,0,0}\put(3781,-3885){\vector(-1, 2){944.600}}
}%
{\color[rgb]{0,0,0}\put(4253,-3885){\vector( 0, 1){1889}}
}%
{\color[rgb]{0,0,0}\put(4725,-3885){\vector( 1, 2){944.600}}
}%
\put(6143,-4830){\makebox(0,0)[b]{\smash{{\SetFigFont{8}{9.6}{\rmdefault}{\mddefault}{\updefault}{\color[rgb]{0,0,0}$F_{a_m}$}%
}}}}
\put(2363,-4830){\makebox(0,0)[b]{\smash{{\SetFigFont{8}{9.6}{\rmdefault}{\mddefault}{\updefault}{\color[rgb]{0,0,0}$F_{a_0}$}%
}}}}
\put(4253,-4830){\makebox(0,0)[b]{\smash{{\SetFigFont{8}{9.6}{\rmdefault}{\mddefault}{\updefault}{\color[rgb]{0,0,0}\ldots \ldots}%
}}}}
\put(4253,-1996){\makebox(0,0)[b]{\smash{{\SetFigFont{8}{9.6}{\rmdefault}{\mddefault}{\updefault}{\color[rgb]{0,0,0}\ldots \ldots}%
}}}}
\put(5198,-1996){\makebox(0,0)[lb]{\smash{{\SetFigFont{8}{9.6}{\rmdefault}{\mddefault}{\updefault}{\color[rgb]{0,0,0}$F_{a_n}$}%
}}}}
\put(3308,-1996){\makebox(0,0)[rb]{\smash{{\SetFigFont{8}{9.6}{\rmdefault}{\mddefault}{\updefault}{\color[rgb]{0,0,0}$F_{a_{m+1}}$}%
}}}}
\end{picture}%
\end{figure}
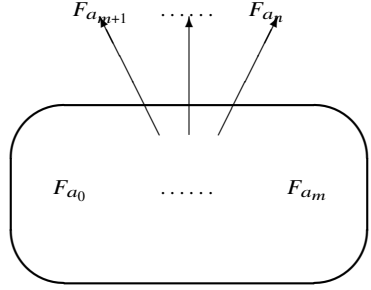
It is easy to see that $\klam{\ult(B),R}$ satisfies the conditions
\begin{xalignat}{2}
&(\forall F)(\exists G)[F\mathrel{R}G \tand (\forall H)(G\mathrel{R}H \Implies G = H)], \label{M} \\
& (F\mathrel{R}G \tand G\mathrel{R}H) \Implies (G\mathrel{R}F \tor H\mathrel{R}G).   \label{B2}
\end{xalignat}
Quasiorders that satisfy \eqref{M} and \eqref{B2} have depth two, and consist of a single cluster on level one and simple clusters on level two. If $U$ is the lower level and $V$ is the upper level, then $R = (U \times U) \cup (U \times V) \cup 1'$. We see that $\mathrel{R}\restrict U^2$ is the universal relation on $U$ and $\mathrel{R}\restrict V^2$ is the identity on $V$. Based on this observation we call a relation of this form a \emph{ui-relation}. If $\klam{B,f}$ is a filter algebra we call $f$ a \emph{ui-operator} and denote it by $f^{\text{ui}}$. Note that a ui-relation is the converse of an iu-relation, and thus, their operators are conjugate functions in the sense of~\cite{jt51}.

According to~\cite[p. 45f]{bs84} the reflexive and transitive frames that satisfy \eqref{M} and \eqref{B2} are those that validate the logic $\mathbf{S4MB_2}$. Conversely, it is easy to show that the complex algebra of a ui--relation is a filter algebra. Thus, applying Corollary~\ref{cor:ui}, we obtain
\begin{theorem}\label{thm:logfmoa}
$\Eq(\fmoa) = \Eq(\mathbf{S4MB_2})$.
\end{theorem}
Owing to the relation of the canonical relations we may say that in some sense the logic $\mathbf{S4.3DumB_2}$ of ideal algebras is the converse of the logic $\mathbf{S4MB_2}$ of filter algebras.

\section{{uu}-- and ii--algebras and their logics}\label{sec:uuii}

In this section we consider the two remaining types of closure algebras of depth at most two which may be related to ideals. Suppose that $a \neq 0$, and that $f$ is a closure operator on $B$ such that $B^{\text{c}} = \set{0,a,1}$. Then,
\begin{gather*}
f(x) \df
\begin{cases}
0, &\text{if } x = 0, \\
a, &\text{if } 0 \lneq x \leq a, \\
1, &\text{otherwise}.
\end{cases}
\end{gather*}
Clearly, $f$ is a closure operator. Considering the ideal $I \df \da{a}$, we see that $f$ maps the nonzero elements of $I$ to its maximum, and therefore we call $\klam{B,f}$ a \emph{MaxId algebra}. Technically, we could also allow $a = 0$, i.e. $I = \set{0}$, for a MaxId algebra. In this case $f$ is the unary discriminator, and $f$ can also be described by $a = 1$. So, we shall always assume $a \gneq 0$. The class of MaxId algebras is denoted by \mmoa.
\begin{theorem}\label{thm:uu}
$\mmoa$ is a positive universal class.
\end{theorem}
\begin{proof}
We will show that a closure algebra $\klam{B,f}$ with $\card{B} \geq 2$ is a MaxId--algebra  \tiff it satisfies
\begin{align}
& (\forall x,y)[(0 \lneq f(x), f(y) \lneq 1) \Implies f(x) = f(y)], \label{uu2} \\
& (\forall x,y)[(y \neq 0 \tand x \not\leq f(y)) \Implies f(x) = 1]. \label{uu3}
\end{align}
Clearly, \eqref{uu2} and \eqref{uu3} are equivalent to positive universal sentences in disjunctive form.

\aright If $a = 1$, then \eqref{uu2} and \eqref{uu3} are vacuously satisfied. Suppose that $a \neq 1$; then, $f$ has exactly the values $\set{0,a,1}$. If $0 \lneq f(x), f(y) \lneq 1$, we have $f(x) = f(y) = a$, and \eqref{uu2} is satisfied.

Suppose that $y \neq 0 \tand x \not\leq f(y)$; then, $x \neq 0$ and $f(y) \neq 1$. Since $0 \neq y \leq f(y) \lneq 1$, we have $f(y) = a$, and $x \not\leq f(y)$ implies $f(x) = 1$.

\aleft Suppose that $f$ is a closure operator that fulfils \eqref{uu2} and \eqref{uu3}. If $\Bc = \set{0,1}$, then the fact that $f$ is expanding implies that $f$ is the unary discriminator, and we set $a \df 1$. Otherwise, there is some $b \in B$ such that $0 \lneq f(b) \lneq 1$, and we set $a \df f(b)$. This is well defined by \eqref{uu2}. Let $x \neq 0$. There are two cases:
\begin{enumerate}
\item $x \leq f(b)$: Then $0 \lneq f(x)  \leq f(f(b)) = f(b) \lneq 1$, and thus, $f(x) = f(b)$ by \eqref{uu2}.
\item $x \not\leq f(b)$: Then, $f(b) \neq 0$ implies that $b \neq 0$, and thus, $f(x) = 1$ by \eqref{uu3}.
\end{enumerate}
This completes the proof.
\end{proof}

Theorem \ref{thm:uu} implies that $\mmoa = \hom\sub\prodsu(\mmoa)$, and applying J{\'o}nsson's Lemma we obtain

\begin{corollary}\label{cor:sifmoa}
If $\klam{B,f} \in \Eq(\mmoa)$ is subdirectly irreducible, then $\klam{B,f} \in \mmoa$.
\end{corollary}
%

\begin{theorem}\label{thm:si-mmoa}
Each MaxId algebra $\klam{B,f}$ with $\card{B} \geq 2$ is subdirectly irreducible.
\end{theorem}
\begin{proof}
If $a$ is the smallest element of $B^{\text{c}} \setminus\set{0}$, then $\da{a}$ is the smallest closed nontrivial ideal of $B$, and thus, $B$ is subdirectly irreducible.
\end{proof}

\begin{theorem}
$\Eq(\maxid)$ is locally finite.
\end{theorem}
\begin{proof}
For any algebra $B$ in \maxid, $\Bo$ has at most three open elements, and thus, $K_3$ cannot be embedded in it. Now use Theorem~\ref{thm:locfin}.
\end{proof}

Next, we consider the canonical frames of MaxId algebras.

\begin{theorem}\label{thm:maxid}
\begin{enumerate}
\item The canonical frame of a MaxId algebra is a chain of length at most two of two clusters.
\item The complex algebra of a chain of at most two clusters is a MaxId algebra.
\end{enumerate}
\end{theorem}
\begin{proof}
1.  Let $\klam{B,f}$ be a MaxId algebra determined by $a$, and set $U \df \set{F \in \ult(B): a \in F}$, and $V \df \set{F \in \ult(B): a \not\in F}$.  Let $F,G \in \ult(B)$, and recall that $F \mathrel{R_{f}} G$ \tiff $f[G] \subseteq F$. First, suppose that $f$ is the unary discriminator, i.e. that $a = 1$. Then, $V = \z$ and $R_f$ is the universal relation on $\ult(B)$, which shows that  the canonical frame consists of one cluster. Otherwise, suppose that $0 \lneq a \lneq 1$; then neither $U$ nor $V$ are empty. Suppose that $F \in U$. Then, $f[G] \subseteq \set{a,1} \subseteq F$, and thus, $F\mathrel{R_f}G$; this implies  $U \times \ult(B) \subseteq R_f$. If $F \in V$, i.e. $a \not\in F$, then $f[G]  \subseteq F$ \tiff $a \not\in G$, which shows that $V \times V \subseteq \mathrel{R_f}$, and $(V \times U)\; \cap \mathrel{R_f} = \z$.  Altogether, $R_{f} = (U \times U) \cup (U \times V) \cup (V \times V)$, and so $R$ has the desired form.

2. Suppose that $\klam{W,R}$ is a chain of at most two clusters. If $R$ has just one cluster, then $R$ is the universal relation on $W$, and $\poss{R}$ is the unary discriminator; hence, $\klam{2^W, \poss{R}}$ is a MaxId algebra. Next, suppose that $R$ has two levels, $U,V$, and $R = (U \times U) \cup (U \times V) \cup (V \times V)$. We will show that $\poss{R}(X) \in \set{\z, U, W}$ for all $X \subseteq W$. Since $\klam{R}$ is completely additive, it is sufficient to consider singletons. Let $x \in U$; then, $\poss{R}(\set{x}) = \conv{R}(x) = U$. If $x \in V$, then $\conv{R}(x) = W$. Hence $\poss{R}$ is a MaxId operator.
\end{proof}
It follows that $\Eq(\mmoa)$ is generated by the complex algebras of finite frames of depth at most two, each level of which contains one cluster. It is well known from e.g.~\cite[Theorem 1]{Fine} or~\cite[p. 52]{bs84} that frames of this form validate the logic $\mathbf{S4.3B_2}$, so that we obtain
\begin{theorem}\label{thm:logmmoa}
$\Eq(\mmoa) = \Eq(\mathbf{S4.3B_2})$.
\end{theorem}

If $\klam{B,f}$ is a MaxId--algebra we shall indicate this by writing $f^{\text{uu}}$ for $f$, where the superscript uu indicates that the restriction of the canonical relation $R_f$ to a level is the universal relation. We note in passing that for  $a \in B$, $a \not\in \set{0,1}$, and the closure operators $f^{\text{iu}}, f^{\text{ui}}$ and $f^{\text{uu}}$ associated with the principal ideals and filters given by $a$, we have $f^{\text{iu}}(x) + f^{\text{ui}}(x) = f^{\text{uu}}(x)$, and $R_{f^{\text{iu}}} \cup R_{f^{\text{ui}}} = R_{f^{\text{uu}}}$.

Thus far we have considered algebras whose associated frames have depth at most two and in which at least one level contains a proper cluster. We will now consider the remaining case where each level contains only simple clusters. In accordance with our previous procedure, we shall approach this from an algebraic viewpoint first. Let $b \in B$ and let $f$ be defined as follows:

\begin{gather}\label{ii}
f(x) \df
\begin{cases}
x, &\text{if } x \leq b, \\
b+x, &\text{if } x \not\leq b.
\end{cases}
\end{gather}

Clearly, $f$ is a closure operator, and $B^{\text{c}} = \da{b} \cup \ua{b}$. Note that $f$ is the identity \tiff $b \in \set{0,1}$. The class of algebras $\klam{B,f}$ that satisfy \eqref{ii} is denoted by \gmoa. Clearly, \gmoa\ is first order axiomatizable.

\begin{theorem}
$\Eq(\gmoa)$ is locally finite.
\end{theorem}
\begin{proof}
This is analogous to the proof of Theorem~\ref{thm:flf}, using $K_4$ instead of $K_3$.
\end{proof}
Unlike the classes we have considered previously, \gmoa\ is not closed under subalgebras, i.e.\ it is not a universal class:%
\footnote{Our original claim was that \gmoa\ is universal, however, the reviewer pointed out a hole in the proof which could not be fixed, and which led to Example \ref{ex:ii} and subsequent results.}
\begin{example}\label{ex:ii}
Suppose that $B$ is an atomless Boolean algebra, $b \in B \setminus \set{0,1}$, and let $f$ be determined by $b$. Choose some non-principal ideal $I \subsetneq \da{b}$, and set $A \df I \cup -I$; then, $A$ is a Boolean subalgebra of $B$ \cite[Lemma 5.33]{kop89}, and $I$ is a prime ideal of $A$. Let $g \df f \restrict A$; we show that $g[A] \subseteq A$. If $x \in I$, then $x \leq b$, and therefore $g(x) = f(x) = x \in I$. If $x \in -I$, then $-b \leq x$, and $g(x) = f(x) = b + x = 1$. Thus, $\klam{A, g}$ is a subalgebra of $\klam{B,f}$ but it is not in \gmoa, since $I$ is non-principal.
\end{example}
Note that $\klam{A,g}$ above is an ideal algebra. The following observation shows that the form of $A$ in the previous example is no accident:
\begin{theorem}\label{thm:sfin}
Let $\klam{B,f} \in \gmoa$ with respect to $b$ and let $\klam{A,g}$ be a subalgebra of $\klam{B,f}$. Then, $b \in A$ and $\klam{A,g} \in \gmoa$, or there is a prime ideal $J$ of $A$ such that $J \subseteq \da{b}$, $A = J \cup -J$ and $\klam{A,J} \in \ima$.
\end{theorem}
\begin{proof}
If $b \in A$, then clearly $\klam{A,g} \in \gmoa$. Thus, suppose that $b \not\in A$. Assume that there is some $x \in A$ such that $x \not\leq b$ and $-x \not\leq b$. Then,
\begin{gather*}
g(x) \cdot g(-x) = f(x) \cdot f(-x) = (b+x) \cdot (b + -x) = b,
\end{gather*}
contradicting $b \not\in A$. Thus, $x \lneq b$ or $-x \lneq b$ for all $x \in A$. Set $J \df \set{x \in A: x \leq b}$; then, $J$ is a closed prime ideal of $A$ since $g\restrict J$ is the identity, and $A = J \cup -J$. Finally, if $x \in A \setminus J$, then $x \not\leq b$, and therefore, $-x \leq b$, and $g(x) = f(x) = x + b \geq x + (-x) = 1$. Hence, $\klam{A,J} \in \ima$.
\end{proof}
%
%
\begin{corollary}
Let $\klam{B,f} \in \gmoa$ and let $\klam{A,g}$ be a finite subalgebra of $\klam{B,f}$. Then, $\klam{A,g} \in \gmoa$.
\end{corollary}
\begin{proof}
If $b \not\in A$, choose $a \df \max~J$ in the proof above which exists and belongs to $J$ since $A$ is finite; then, $\klam{A,g} \in \gmoa$ with respect to $a$. Indeed, if $x \leq a$, then $x \leq b$, since $a \in J$, and therefore, $g(x) = f(x) = x$. Otherwise, $x \not\in J$ implies $-x \in J$, since is prime, and therefore, $-x \leq a$; hence $x + a = 1$. Furthermore, $x \not\in J$ implies $x \not\leq b$, and therefore, $g(x) = f(x) = x + b = 1 = x+a$.
\end{proof}

\begin{theorem}\label{thm:gmoapu}
\gmoa\ is a positive class.
\end{theorem}
\begin{proof}
We show that \gmoa\ is closed under homomorphic images. Let $\klam{B,f} \in \gmoa$ with respect to $b$, $I$ be a congruence ideal of $B$, and suppose \wlg that $I$ is nontrivial and proper. Let $A \df B/I$, $\pi: B \onto A$ be the canonical homomorphism, and $g(\pi(x)) \df \pi(f(x))$. We suppose \wlg that $\card{A} \geq 4$. There are two cases:
\begin{enumerate}
\item $b \in I$: Our aim is to show $g$ is the identity, i.e. that $g(\pi(x)) = \pi(x)$ for all $x \in B$. Let $x \in B$. If $x \in I$, then $\pi(x) = 0$ and there is nothing more to show. Suppose that $x \not\in I$; then, in particular, $x \not\leq b$, since $b \in I$ and $I$ is an ideal. Now,
\begin{xalignat*}{2}
y \in g(\pi(x)) &\Iff y \in \pi(f(x)), \\
&\Iff (\exists z \in I)[y + z = f(x) + z], \\
&\Iff (\exists z \in I)[y + z = b + x + z], && \text{since } x \not\leq b, \\
&\Iff (\exists z' \in I)[y + z' = x + z'], && \text{set }z' \df z + b \in I,\\
&\Iff y \in \pi(x).
\end{xalignat*}
\item $b \not\in I$: Then, $I \subsetneq \da{b}$ since $I$ is closed, and $B^{\text{c}} = \da{b} \cup \ua{b}$; furthermore, $\pi(b) \neq 0$,  Suppose that $\pi(x) \leq \pi(b)$; then, there is some $y \in I$ such that $x + y \leq b +y$. Since $I \subseteq \da{b}$, it follows that $y \leq b$, and therefore, $x \leq b$ which implies $f(x) = x$; hence, $g(\pi(x)) = \pi(f(x)) = \pi(x)$. If $\pi(x) \not\leq \pi(b)$, then, in particular, $x \not\leq b$ which implies $f(x) = b + x$. It follows that $g(\pi(x)) = \pi(f(x)) = \pi(b) + \pi(x)$.
\end{enumerate}
This completes the proof.
\end{proof}
To describe the subdirectly irreducibles in $\Eq(\gmoa)$, we first describe those in \gmoa.
\begin{lemma}\label{thm:si-gmoa}
Suppose that $\klam{B,f} \in \gmoa$ with respect to $b$. Then, $\klam{B,f}$ is subdirectly irreducible \tiff $\card{B} = 2$ or $b$ is an atom of $B$.
\end{lemma}
\begin{proof}
\aright Suppose that $\card{B} \geq 4$. If $b \in \set{0,1}$, then $f = 1'$, and  $\klam{B,f}$ is not subdirectly irreducible. Hence, $0 \lneq b \lneq 1$, and, $\da{b}$ is a nontrivial closed ideal. If $b$ is not an atom, then there are $0 \lneq x,y \lneq b$ such that $x \cdot y = 0$. Then, $\da{x}$ and $\da{y}$ are nontrivial closed ideals of $B$ strictly contained in $\da{b}$. On the other hand, $\da{x} \cap \da{y} = \set{0}$, and therefore, $B$ is not subdirectly irreducible.

\aleft If $\card{B} = 2$, then  $\klam{B,f}$ is simple, hence, subdirectly irreducible. If $b$ is an atom of $B$, then $\da{b}$ is the smallest nontrivial closed ideal of $B$, since $f(x) = x$ \tiff $x \leq b$ or $b \leq x$.
\end{proof}
\begin{theorem}\label{cor:sigmoa}
If $\klam{B,f} \in \Eq(\gmoa)$ is subdirectly irreducible, then $\klam{B,f} \in \gmoa$.
\end{theorem}
\begin{proof}
 Suppose that $\klam{B,f} \in \Si\Eq(\gmoa)$. By J{\'o}nsson's Lemma \ref{lem:jon}, $\klam{B,f} \in \hom\sub(\gmoa)$. It is well known that $\hom\sub(\K) = \sub\hom(\K))$ for any class $\K$ of modal algebras  \cite{blok80a}, and therefore, $\hom\sub(\gmoa) = \sub(\gmoa))$ by Theorem \ref{thm:gmoapu}. Assume that  $\klam{B,f} \not\in \gmoa$. Then, $\card{B} \geq 4$, and there is some $\klam{B',f'} \in \gmoa$ such that $\klam{B,f}$ is a subalgebra of $\klam{B',f'}$. By Theorem \ref{thm:sfin}, there is some closed prime ideal $I$ of $B$ such that $B = I \cup -I$. Since $\klam{B,f}$ is subdirectly irreducible, $I$ must be generated by an atom by Theorem \ref{thm:idalg}, say, $b$. Then, $B$ is the four element Boolean algebra with $B^{\text{c}} = \set{0,b,1}$ since $B = I \cup -I$, and thus, $\klam{B,f} \in \gmoa$. This contradicts our assumption.
\end{proof}

Therefore, since $\Eq(\gmoa)$ is locally finite, it is generated by the finite subdirectly irreducible algebras in \gmoa.

Next we turn to the canonical frames of algebras in \gma.

\begin{theorem}\label{thm:iiframe}
\begin{enumerate}
\item Suppose that $\klam{B,f} \in \gmoa$ with respect to $a$, and that $a \not\in \set{0,1}$. Then $R_{f}$ has two levels $U,V$, and $R_{f} = (U \times V) \cup 1'$.
\item If $\klam{W,R}$ is a frame where $R$ has two levels $U,V$ and $R = (U \times V) \cup 1'$, then $\klam{R}$ satisfies \eqref{ii}.
\end{enumerate}
\end{theorem}
\begin{proof}
1. Let $U \df \set{F \in \ult(B): a \in F}$, and $V \df \set{F \in \ult(B): a \not\in F}$, and $F,G \in \ult(B)$.

\sbe Let $F \mathrel{R_f} G$ and $F \neq G$; then, there is some $x \in G$, $x \not\in F$. Assume that $a \in G$. Since $G$ is a filter and $x \in G$, we have $x \cdot a \in G$. By \eqref{ii}, $f(x \cdot a) = x \cdot a$, and $F \mathrel{R_f} G$ implies that $x \cdot a \in F$. Since $F$ is a filter, we obtain $x \in F$, a contradiction. It follows that $G \in V$. Next, assume that $a \not\in F$; then, $x + a \not\in F$, since $x \not\in F$, and $F$ is an ultrafilter. However, $f[G] \subseteq F$ implies that $f(x+a) = x+a \in F$, a contradiction. It follows that $F \in U$, and altogether we have proved $R_f \subseteq (U \times V) \cup 1'$.

\spe We first show that $F \mathrel{R_f}F$, i.e. that $f[F] \subseteq F$. Let $x \in F$. Since $f$ is a closure operator, we have $x \leq f(x)$, which implies $f(x) \in F$. Next, let $F \in U$ and $G \in V$. Suppose that $x \in G$. Since $-a \in G$, we have $x \not\leq a$, and thus, $f(x) = a + x$. Now, $a \in F$ implies $a + x \in F$, and therefore, $F \mathrel{R_f} G$. It follows that $(U \times V) \cup 1' \subseteq R_f$.

2. Suppose that $R$ has two levels $U,V$ and $R = (U \times V) \cup 1'$. If $x \in U$, then $\conv{R}(x) = \set{x}$, and if $x \in V$, then $\conv{R}(x) = U \cup \set{x}$. Altogether, $\klam{R}(Y) = Y$, if $Y \subseteq U$, and $\klam{R}(Y) = U \cup Y$ otherwise.
\end{proof}
We call a relation with two distinct levels $U,V$ for which $R = (U \times V) \cup 1'$ an \emph{ii--relation}, indicated by $R^{\text{ii}}$,  and an operator which satisfies \eqref{ii} an \emph{ii--operator}. Example \ref{ex:ii} shows that the canonical extension -- i.e.\ the complex algebra of the ultrafilter extension -- of an ideal algebra \A\ can be in \gmoa, even though \A\ is not.

It follows from the previous results that $\Eq(\gmoa)$ is generated by the complex algebras of ii--relations where the lower level has exactly one element. It is well known that the logic generated by these frames is the pretabular logic $\mathbf {S4 Grz B_2}$ of~\cite[p. 53]{bs84}, see also~\cite{em77} and~\cite{maks75}. It is known that $\mathbf {S4 Grz B_2} = \mathbf {S4MDumB_2}$, see e.g.~\cite[p. 107]{seg71}; in the sequel we shall use $\mathbf {S4MDumB_2}$ instead of $\mathbf {S4 Grz B_2}$, since it demonstrates better the relationship among the logics we have considered. From the previous result we now obtain
\begin{theorem}
$\Eq(\gmoa) = \Eq(\mathbf {S4MDumB_2})$.
\end{theorem}
 Figures \ref{fig:2level} and \ref{fig:2linc} summarize the corresponding frame relations of depth two:
\begin{description}
\item[$R^{\text{ii}}$:] Simple clusters on both levels. If the lower level contains just one simple cluster, the complex algebra is subdirectly irreducible. The corresponding logic is the pretabular logic $\mathbf {S4MDumB_2}$, and $\Eq(\gmoa)$ is the class of its algebraic models.
\item[$R^{\text{iu}}$:]  Simple clusters on the lower level, one cluster on the upper level. If the lower level contains just one simple cluster, the complex algebra is subdirectly irreducible. The corresponding logic is $\mathbf{S4.3DumB_2}$, and $\Eq(\imoa)$ is the class of its algebraic models.
\item[$R^{\text{ui}}$:] One cluster on the lower level, simple clusters on the upper level. Since such a frame is rooted, its complex algebra is subdirectly irreducible. The corresponding logic is $\mathbf{S4MB_2}$, and $\Eq(\fmoa)$ is the class of its algebraic models.
\item[$R^{\text{uu}}$:] One cluster on the lower level, one cluster on the upper level. Since such a frame is rooted, its complex algebra is subdirectly irreducible. The corresponding logic is $\mathbf{S4.3B_2}$, and $\Eq(\mmoa)$ is the class of its algebraic models.
\end{description}

\begin{figure}[h!tb]
\begin{minipage}[t]{0.6\textwidth}
\caption{The relations of depth two}\label{fig:2level}
\vspace{3mm}
\centering
\includegraphics[width=0.9\textwidth]{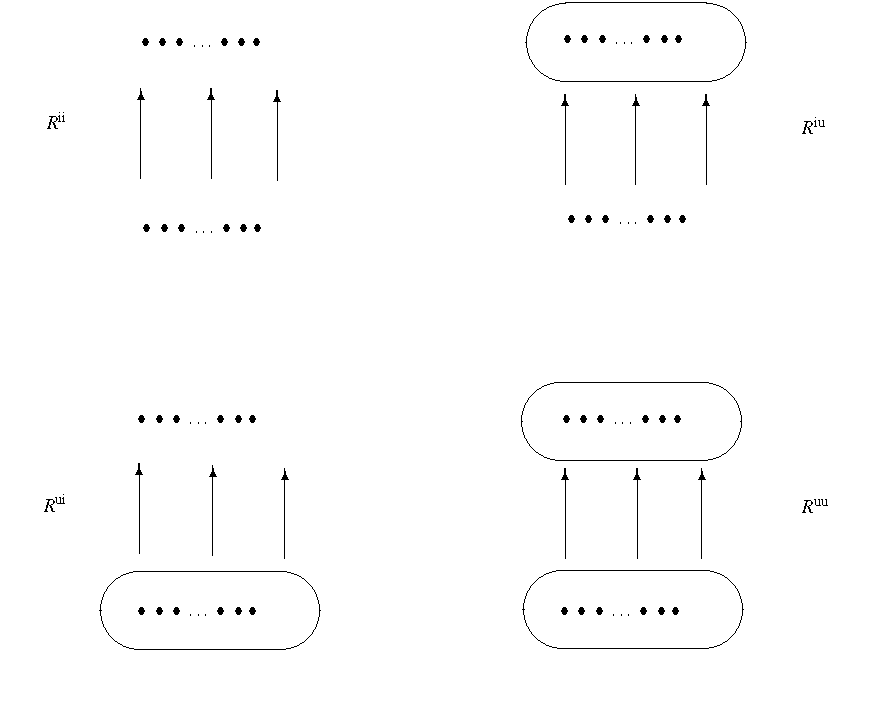}
\end{minipage}
\begin{minipage}[t]{0.33\textwidth}
\caption{Form and inclusion of the relations}\label{fig:2linc}
\vspace{3mm}
\centering
\begin{tabular}{cl}
Type & $U$ level 1, $V$ level 2 \\ \hline
ii & $(U \times V) \cup 1'$ \\
iu & $(U \times V) \cup (V \times V) \cup 1'$ \\
ui & $(U \times V) \cup (U \times U) \cup 1'$ \\
uu & $(U \times V) \cup (U \times U) \cup (V \times V$)
\end{tabular}

  $$
\xymatrix{
& R_a^{\text{uu}} & \\
R_a^{\text{iu}} \ar@{->}[ru]^{\subseteq}  && R_a^{\text{ui}} \ar@{->}[lu]_{\supseteq} \\
& R_a^{\text{ii}} \ar@{->}[ru]_{\subseteq} \ar@{->}[lu]^{\supseteq} &
}
$$
\end{minipage}
\end{figure}

\section{Meet and join of classes of algebras of depth two}\label{sec:meet}

The four varieties we have considered are locally finite with the subdirectly irreducibles contained in the generating class, and they all contain the two element closure algebra. Table \ref{tab:clsi} lists the algebra classes, the corresponding logics, and the closed elements of the subdirectly irreducibles when $\card{B} \geq 4$.
\begin{table}[h!tb]
\caption{Algebras, logics, subdirectly irreducibles}\label{tab:clsi}
\begin{center}
\begin{tabular}{llll}
\multicolumn{1}{c}{Class} & Logic & \multicolumn{2}{c}{Closed elements when $f \neq f^{\one}$} \\ \hline
\gmoa & $\mathbf {S4MDumB_2}$ & $\set{0} \cup \ua{b}$, $b$ an atom & (Theorem~\ref{thm:si-gmoa}). \\
\fmoa & $\mathbf{S4MB_2}$ &$\set{0} \cup \ua{b}$ & (Theorem~\ref{thm:logfmoa}). \\
\imoa & $\mathbf{S4.3DumB_2}$ &$\set{0,a,1}$, $a$ an atom & (Theorem~\ref{thm:idalg}). \\
\mmoa & $\mathbf{S4.3B_2}$ &$\set{0,a,1}$ & (Theorem~\ref{thm:si-mmoa})
\end{tabular}
\end{center}
\end{table}
We see that $\Si(\gmoa) \subseteq \Si(\fmoa)$ and $\Si(\imoa) \subseteq \Si(\mmoa)$, which implies $\Eq(\gmoa) \subseteq \Eq(\fmoa)$ and $\Eq(\imoa) \subseteq \Eq(\mmoa)$.

In the sequel we shall suppose that $\klam{B,f}$ is finite with at least four elements unless stated otherwise. To simplify notation  we shall index the operators with a subscript that indicates the special element generating the associated filter or ideal, depending on the superscript; for example, $f^{\text{iu}}_a$ is the ideal algebra with associated ideal $\da{a}$, and $f^{\text{ui}}_b$ is the filter algebra with associated filter $\ua{b}$. As a preparation for the description of the meet of $\Eq(\mmoa)$ and $\Eq(\fmoa)$ we observe the following:
\begin{lemma}\label{lem:uuui}
If $\klam{B,f} \in \mmoa \cap \fmoa$ is subdirectly irreducible such that $f = f_a^{\text{uu}} = f_b^{\text{ui}}$ and $a,b \not\in \set{0,1}$, then $b$ is an antiatom, and $a = b$.
\end{lemma}
\begin{proof}
Suppose that $\klam{B,f} \in \mmoa \cap \fmoa$, $f = f_a^{\text{uu}} = f_b^{\text{ui}}$.  Since $0 \lneq a$ and $b \lneq 1$, $f$ is neither the identity nor the unary discriminator. Thus, $\Bc = \set{0,a,1} = \set{0} \cup \ua b$ which implies $\set{a,1} = \ua b$, hence, $a = b$, and $b$ is an antiatom.
\end{proof}

The canonical relation of an algebra satisfying the conditions of Lemma~\ref{lem:uuui} has two levels with one cluster on the lower level, and a simple cluster on the second level. The logic determined by this type of frame is the pretabular logic $\mathbf{S4.3MB_2}$, see~\cite[p. 53, $\klam{W_2,R_2}$]{bs84}. Thus, we obtain
\begin{theorem}
$\Eq(\mmoa) \land \Eq(\fmoa) = \Eq(\mathbf{S4.3MB_2})$.
\end{theorem}

This also follows from Theorems \ref{thm:logfmoa} and \ref{thm:logmmoa}, since the lattice of normal modal logics is dually isomorphic to the lattice of equational classes of modal algebras, see e.g.~\cite{blok80a}.

If $B$ is generated by its atoms $a,b$, and $I = \da{a}$, then $B$ has two ultrafilters $F_a,F_b$ and $F\mathrel{R_f}G$ \tiff $I \cap G \subseteq F$~\cite{dd21}. If $G = F_b$, then $I \cap G = \z$, and it follows that $F_b\mathrel{R_f}F_b$ and also $F_a \mathrel{R_f} F_b$. Furthermore, $\set{a} \subseteq F_a$ implies $F_a \mathrel{R_f} F_a$. On the other hand, $\set{a} = I \cap F_a \not\subseteq F_b$, and thus, $F_b(-\mathrel{R_f})F_a$. This shows that the canonical relation of $\klam{B,f}$ has two levels each containing a simple cluster; the element in the lower level is $R_f$-related to the element on the upper level, but not vice versa; in other words, $R_f$ is the two element chain. Let us denote a frame of this type by $\F_2$, and its complex algebra by $\F_2^+$. Note that $\F_2^+$ is subdirectly irreducible, but not simple. Furthermore, the set of closed elements of $\F_2^+$ is a chain of length three. Conversely, if $\card{B} = 4$, and $\Bc$ has three elements, then $\klam{B,f} \cong \F^+_2$. Clearly, $\F^+_2$ is in all four classes of algebras which we have considered.

\begin{theorem}
$\Eq(\gmoa) \land \Eq(\mmoa) = \Eq(\F_2^+)$.
\end{theorem}
\begin{proof}
Since $\F^+_2 \in \Eq(\gmoa) \cap \Eq(\mmoa)$, we need only show the ``$\subseteq$'' direction. Suppose that $\klam{B,f} \in \Eq(\gmoa) \land \Eq(\mmoa)$ is subdirectly irreducible. Indeed, by Corollaries \ref{cor:sifmoa} and \ref{cor:sigmoa} $\klam{B,f}$ belongs to \gma\ and \mma. If $\card{B} = 2$, then it is isomorphic to a subalgebra of $\F_2^+$. Otherwise, $\klam{B,f} \in \Eq(\gmoa)$ implies that there is an atom $b$ of $B$ such that $b \neq 1$ and $\Bc = \set{0} \cup \ua{b}$ by Theorem~\ref{thm:si-gmoa}.  Thus, $\Bc$ has at least three elements. On the other hand, $\klam{B,f} \in \mmoa$ implies that $\Bc$ has at most three elements, and therefore, $\Bc$ has exactly three elements. This implies that $\card{\ua{b}} = 2$, hence, $\card{B} = 4$, since $b$ is an atom. Thus, $\klam{B,f} \cong \F^+_2$.
\end{proof}
Similarly, we can show:
\begin{theorem}
$\Eq(\ima) \land \Eq(\fma) = \Eq(\F_2^+)$.
\end{theorem}
It is well known, e.g. from~\cite{zeman71}, that the logic determined by $\Eq(\F_2^+)$ is the logic $\mathbf{S4.3MDumB_2}$, called \emph{K4} by Soboci{\'n}ski~\cite{Sobocinski1964}, not to be confused with the nowadays common notation \textbf{K4} for the class of normal logics determined by transitive frames.

The meet sub--semilattice of the lattice of modal varieties generated by the four varieties we have considered is shown in Figure \ref{fig:meet}.

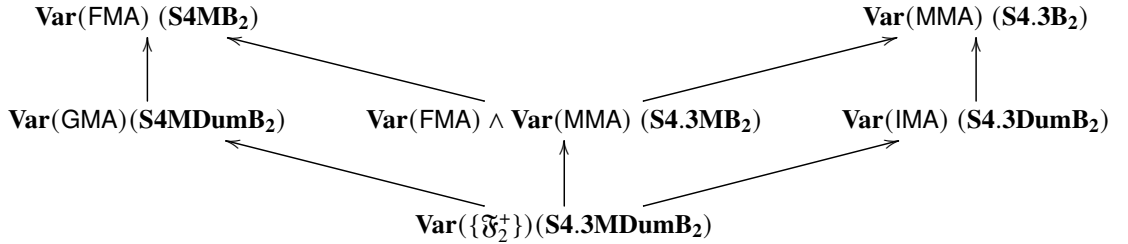
\begin{figure}[htb]
\caption{The meets of classes generated by algebras of depth two and extremal frame relations}\label{fig:meet}
$$
\xymatrix{
\Eq(\fmoa) \ (\mathbf{S4MB_2})  &  & \Eq(\mmoa)  \ (\mathbf{S4.3B_2})\\
\Eq(\gmoa) \ar@{->}[u] (\mathbf{S4MDumB_2}) &\Eq(\fmoa) \land \Eq(\mmoa) \ (\mathbf{S4.3MB_2}) \ar@{->}[ru] \ar@{->}[lu] &
\Eq(\imoa)  \ (\mathbf{S4.3DumB_2}) \ar@{->}[u] \\
 &  \Eq(\set{\F_2^+})  (\mathbf{S4.3MDumB_2}) \ar@{->}[lu] \ar@{->}[ru] \ar@{->}[u]
}
$$
\end{figure}

Some remarks on pretabularity are in order. $\imoa$  has as an important subclass the class of discriminator algebras  $\dia$,  (corresponding to the ideal $\{0\}$). If we were to add to Figure \ref{fig:meet} the class $\Eq (\dia)$ belonging to \textbf{S5}  in the lower right hand corner, we realize that the classes of closure algebras that we study include all three types of  finite closure algebras belonging to pretabular logics of depth at most two: $\mathbf{S5}$ (\emph{clots}), $\mathbf{S4.3MB_2}$ (\emph{tacks}), and $\mathbf{S4MDumB_2}$, i.e. $\mathbf{S4Grz B_2}$ (\emph{fans}). We do not capture the remaining two types: \emph{Chains}, which may have  any finite depth, and \emph{tops} which have depth three, see e.g. \cite{em77}.

For the meet of the logics (or join of corresponding varieties)  we apply Theorem~3 from~\cite{zwc01}, adapted for \texttt{S4} logics. If $\phi, \psi$ are modal formulas we denote by $\phi \ulor \psi$ the formula obtained by replacing the variables in $\varphi \lor \psi $ in such a way that $\varphi $ and $\psi$ have no variables in common.
\begin{theorem}\label{thm:lmeet}
Suppose that $\mathbf{L}_1$ and $\mathbf{L}_2$ are \texttt{S4} logics, say, $\mathbf{L}_1 = \K \cup  \{\phi_i :  i \in I \}$, $\mathbf{L}_2 = \K \cup  \{\psi_j :  j \in J \}$. Then,
\begin{gather}\label{lmeet}
\mathbf{L}_1 \wedge \mathbf{L}_2 = \K \cup \{ \Box \phi_i  \ulor \Box \psi_j  :  i \in I ,  j \in J \}.
\end{gather}
\end{theorem}
Using Theorem~\ref{thm:lmeet} we obtain
\begin{corollary}
\begin{enumerate}
\item $\mathbf{S4MB_2} \wedge \mathbf{S4.4} = \mathbf{S4B_2 (\Box M \ulor \Box .4)}$.
\item $\mathbf{S4.3MB_2} \land \mathbf{S4.3DumB_2} = \mathbf{S4.3B_2(\Box M \ulor \Box Dum)}$.
\item $\mathbf{S4.3DumB_2} \land \mathbf{S4MDumB_2} = \mathbf{S4DumB_2(\Box M \ulor \Box .3)}$.
\item $\mathbf{S4MB_2} \land \mathbf{S4.3B_2} = \mathbf{S4B_2 (\Box M \ulor \Box .3)}$.
\end{enumerate}
\end{corollary}
By Theorem~\ref{lem:423}, \textbf{S4.3} may be replaced by \textbf{S4.2}. The subdirectly irreducibles of the corresponding equational classes of modal algebras may be obtained by Lemma~\ref{lem:si}(2).

Including or excluding some particular ``boundary conditions'' may give rather  unexpected changes in the resulting varieties and logics. For instance, if in the definition of filter algebras (which form the class \fmoa ) one allows $F$ to be equal to $\set{1}$, we obtain the class $\fmoa \vee \dia$,
whose meet $\fmoa \wedge \dia$ contains only the closure algebra with at most two-elements, in terms of logic,  $\mathbf{S4MB_2} \vee \mathbf{S5} = \mathbf {Triv}$. Moreover, the logic $\mathbf{S4MB_2} \land \mathbf{S5}$   determined by $ \fmoa \vee \dia $ differs from  $\mathbf{S4MB_2}$ substantially. Also, Figure \ref{fig:meet} becomes much more complicated.


\section{Quasivarieties of depth two algebras: Structural completeness}\label{sec:SC}

In this section  we consider rules of inference in some of the logics we have considered, their related algebras, and their quasivarieties. In general, structural completeness concerns a deductive system $\L $  (where $\L$  can be determined by a set of axioms and a set of rules of inference)  or its consequence operation $\vdash_\L$, not just a logic understood as a set of formulas closed with respect to some rules.  A deductive system \L \ (or its consequence operation $\vdash_\L$)   is called {\it structurally complete} (SC) if every admissible rule in \L \ is also derivable in \L\ . The set of theorems of a system \L \; is denoted by $\Log$, hence,
\begin{center}
 $\vdash_\L \phi$ if and only if $\phi$ is derivable in $\L$  if and only
if $\phi \in \Log$.
 \end{center}
 A consequence operation $\vdash $ is SC \tiff  $\vdash $ is maximal among all $\vdash '$ such that:  $\emptyset \vdash' \psi$ \tiff  $\emptyset \vdash  \psi$ for all $\psi$, i.e. if they have the same set of theorems.
  A rule $r: \phi_1, \dots, \phi_n \slash \psi$  is {\it passive} if   for every substitution $\varepsilon$, $\{\varepsilon \phi_1, \dots, \varepsilon \phi_n \} \not \subseteq  \Log$. For example, the rule $P_2$
\begin{gather}\label{p2}
\Diamond \phi \land \Diamond \neg \phi \slash \psi
\end{gather}
or, equivalently, $\Diamond \phi \land \Diamond \neg \phi \slash \bot$, is passive, hence admissible (but not derivable  in many modal logics, for
example, in \textbf{S5}). Therefore, we call \L\ {\it almost structurally
complete (ASC)}, if every rule which is admissible and not passive is derivable in \L, see~\cite{ds16}%
\footnote{Recently ``almost structural completeness'' was also called ``active structural completeness'', since passive rules are neglected.}%
. Slightly abusing the terminology, we say that a modal logic $\Log$ is
SC (ASC) if its standard consequence operation, understood as based on the axioms of $\Log$ plus Modus Ponens and the Necessitation rule only, denoted here by $\vdash_\Log$,   is SC (ASC). For instance \textbf{S5} and, indeed, every extension of \textbf{S4.3} is ASC but, in general, not SC, see~\cite{dw12}.

For a variety \K\ of algebras, let $\mathcal F_\K(\lambda)$ be its $\lambda$-generated free algebra; we omit \K\ if no confusion can arise. The following descriptions of SC and ASC (adapted here to closure algebras) are
known, see~\cite{ber88},~\cite{ds16}.

\begin{theorem}\label{algASC}  Let \K\ be a locally finite variety of closure algebras. Then
\begin{enumerate}
\item \K\ is SC \tiff for every finite subdirectly irreducible algebra \A
\ in \K\ , \A \ embeds into  $\mathcal F(\omega)$.
\item \K\ is ASC \tiff for every finite subdirectly irreducible algebra \A \ in \K\ , $\mathfrak A  \times \mathcal F(0)$ embeds into  $\mathcal F(\omega)$.
\end{enumerate}
 \end{theorem}

A substitution  $\varepsilon$ of formulas is called a {\it unifier} for a
formula $\phi$ in the logic $\Log$  if $ \varepsilon \phi\in \Log$. A formula $\phi$ is  {\it unifiable} in  $\Log$, if $\varepsilon \phi \in \Log$ for some substitution $\varepsilon$.
Therefore, a rule $r: \phi_1, \dots, \phi_n \slash \psi$ is passive \tiff
$\phi_1 \land \dots \land \phi_n$ is not unifiable. In case of logics extending $\mathbf {S4}$ it is enough to consider rules of the form $r: \phi
\slash \psi$ only.
A {\it projective unifier} for  $\phi$ in $\Log$  is a unifier $\varepsilon$ for  $\phi$ such that $ \phi\vdash_\Log  \varepsilon(\psi)\leftrightarrow \psi$,  for each formula $\psi$; and one says that a logic $\Log$ enjoys {\it projective unification} if each  $\Log$-unifiable  formula has a projective unifier in $\Log$.
 Projective unifiers (formulas, substitutions) were introduced and extensively used by Silvio Ghilardi in his papers on unification of 1997--2004, see e.g.~\cite{Ghi2,Ghi3}. We have the following result (see~\cite{dw12}):
\begin{theorem}\label{ASC-M}
Let $\Log$  be a logic which enjoys  projective unification. Then $\Log$ is almost structurally complete. If, in addition,  any formula which  is not $\Log$-unifiable is inconsistent in $\Log$, then $\Log$ is structurally complete.
\end{theorem}
\begin{proof}
Let $r: \phi \slash \psi$ be an admissible rule in $\Log$ with a unifiable premise  $\phi $ and let $\varepsilon$ be a projective unifier for $\phi $. Then  $ \varepsilon \phi\in \Log$ and $\phi\vdash_\Log  \varepsilon(\psi)\leftrightarrow \psi$, hence $\phi\vdash_\Log \psi$, i.e. $r: \phi \slash \psi$ is derivable in $\Log$ and $\Log$  is ASC.  Now assume, in addition, that any formula which  is not $\Log$-unifiable is inconsistent in $\Log$ and consider any admissible rule in $\Log$,  $r: \phi \slash \psi$  with a premise  $\phi$ which  is not $\Log$-unifiable.  $\phi$ is then inconsistent in $\Log$  and $\phi\vdash_\Log \psi$, for every formula $\psi$, i.e. $\Log$ is SC.
\end{proof}
 In \cite[5.1--5.2]{DzW11} and in \cite[3.3]{kost18} it is shown that
\begin{lemma}\label{P2-M}
Let $\Log$ be a logic extending $\mathbf {S4}$.  If a formula $\phi$ is not unifiable in $\Log$, then $\phi \vdash_\Log  \Diamond \theta \land \Diamond \neg \theta$, for some formula $\theta$.
\end{lemma}

\begin{corollary}\label{M-P2}
Let $\Log$ be a logic extending $\mathbf {S4}$.
\begin{enumerate}
\item If  $\Log$ enjoys  projective unification and \texttt{M}: $\Box\Diamond p \implies  \Diamond\Box p$ is in $\Log$, then $\Log$ is structurally complete.
\item  If a rule  $r: \phi \slash \psi$ is passive in $\Log$, then $r$ is
derivable by an application of the rule $P_2$.
\end{enumerate}
\end{corollary}
\begin{proof}
1. Assume that $\Log \supseteq \mathbf {S4}$ enjoys projective unification,  \texttt{M}: $\Box\Diamond p \implies  \Diamond\Box p$ is in $\Log$ and $\phi$ is not unifiable in $\Log$. Then, by Lemma~\ref{P2-M},   $\phi \vdash_\Log  \Diamond \theta \land \Diamond \neg \theta$, for some $\theta$. Applying the necessitation rule we obtain $\phi~\vdash_\Log~\Box \Diamond \theta \land \Box \Diamond \neg \theta$, i.e. $\phi \vdash_\Log  \neg
( \Box \Diamond \theta \rightarrow  \Diamond \Box \theta)$, which is the negation of an instance of \texttt{M}. This, together with \texttt{M} $\in \Log$, shows that $\phi\vdash_\Log \psi$ for every formula $\psi$. Finally, Theorem \ref{ASC-M} yields that $\Log$ is SC.

2.  Let $r: \phi \slash \psi$ be passive in $\Log$, i.e. $\phi$ is not unifiable in $\Log$. By Lemma \ref{P2-M},  $\phi \vdash_\Log  \Diamond \theta \land \Diamond \neg \theta$ for some formula $\theta$. Now, by an application of $P_2: \Diamond \theta \land \Diamond \neg \theta \slash \psi $, we obtain  $\phi \vdash_\Log \psi $, i.e.  $r$ is derivable.
\end{proof}
 In~\cite[3.19]{dw12} (see also~\cite{DzW11}) it is shown that
\begin{theorem}\label{main1}  A modal logic $\Log$ containing $\mathbf {S4}$ enjoys projective unification if and only if $\mathbf {S4.3} \subseteq \Log $.
\end{theorem}
Hence, together with~\cite[Corollary 4.2]{dw12},  Theorem \ref{ASC-M} and
Corollary \ref{M-P2} we obtain

\begin{corollary}\label{ASC log.3}
\begin{enumerate}
\item The logics $\mathbf{S4.3 B_2}$ and $\mathbf{S4.3DumB_2}$ as well as
all their extensions enjoy projective unification and are almost structurally complete. In other words, each admissible rule in $\mathbf{S4.3 B_2}$ (in  $\mathbf{S4.3DumB_2}$) is derivable  in $\mathbf{S4.3 B_2}$ (in  $\mathbf{S4.3DumB_2}$) or passive (and then derivable by  the rule $P_2$).
\item The logics $\mathbf{S4.3MB_2}$ and $\mathbf{S4.3MDumB_2}$ as well as all their extensions enjoy projective unification and are structurally complete. In other words, each admissible rule in $\mathbf{S4.3MB_2}$ (in
 $\mathbf{S4.3M DumB_2}$) is derivable  in $\mathbf{S4.3MB_2}$ (in  $\mathbf{S4.3M DumB_2}$).
\end{enumerate}
\end{corollary}

Some further remarks are in order:
\begin{enumerate}
\item Neither $\mathbf{S4.3 B_2}$ nor $\mathbf{S4.3DumB_2}$  are structurally complete (the premise of the rule $P_2$ is consistent).
\item ASC for $\mathbf{S4.3 B_2}$ and $\mathbf{S4.3DumB_2}$  holds only for rules with \emph{finite} premises, in accordance with the standard definition. To show that the extension of ASC for rules with infinite premises does not hold, let us consider the rule with the scheme
\begin{gather}
\frac{ \{\Box(\phi_i \leftrightarrow \phi_j)\implies \phi_0 : 0 < i < j <
\omega \}}{\phi_0}
\end{gather}
The rule is not passive. Observe that for each finite modal algebra $B$,
if  $\phi_0 $ is false in $B$, then $\{\Box (\phi_i \leftrightarrow \phi_j)\implies \phi_0  : 0 < i < j < \omega\}$ is false in $B$; in other words, whenever premises are valid in $B$, the conclusion is also valid in $B$. Since  both $\mathbf{S4.3 B_2}$ and $\mathbf{S4.3DumB_2}$ have the finite model property,  the rule  is admissible in  both logics.  But $p_0$ can be derived neither in $\mathbf{S4.3 B_2}$ nor in $\mathbf{S4.3DumB_2}$ from any finite subset of the set $\{\Box(p_i \leftrightarrow p_j)\implies p_0 : 0 < i < j  < \omega \}$, hence the rule is not derivable.
\end{enumerate}

Recall that for a variety $\mathsf{V}$ of algebras, $\mathcal F_{\mathsf{V}}(\lambda)$ denotes its $\lambda$-generated free algebra. Since in our context $\mathcal F_{\mathsf V}(0) = \2$, we obtain the following Corollary from Theorem  \ref{algASC}, Corollary \ref{ASC log.3},  and\cite{dw12}
\begin{corollary}
\begin{enumerate}
\item  For $\mathsf{V} \in \{\Eq(\ima),  \Eq(\mma) \}$ and  for every finite subdirectly irreducible algebra $B$ \ in $\mathsf{V}$,  $B~\times~\2$
embeds into  $\mathcal F_{\mathsf V}(\omega)$.
\item For $\mathsf{V} \in \{\Eq(\fma) \land \Eq(\mma), \Eq(\gma) \land \Eq(\ima) \}$, every finite subdirectly irreducible algebra $B \in \mathsf{V}$  embeds into  $\mathcal F_{\mathsf V}(\omega)$.
\end{enumerate}
\end{corollary}

Next, we consider (almost) structural completeness in quasivarieties $\Q_\Log $ determined by logics \Log.
A rule   $r: \phi_1, \dots, \phi_n \slash \psi$  in a logic $\Log$ can be translated into  a quasiidentity (quasiequation) in the language  of the class of algebras  corresponding to $\Log$ by $(\phi_1' = 1\; \&  \dots \& \; \phi_n' =1) \Rightarrow \psi' =1$; here, $\&$ and $\Rightarrow$ are  connectives of the metalanguage. In particular, for $r: \phi \slash \psi$ we have $\phi' = 1 \ \Rightarrow \psi' =1$. Under this translation, a rule $r: \phi_1, \dots, \phi_n \slash \psi$ is admissible in $L$ \tiff the quasiidentity $(\phi_1' = 1\; \&  \dots \& \; \phi_n' =1) \Rightarrow \psi' =1$ holds in the free algebra  $\mathcal F_{\mathsf{V_L}}(\omega)$, that is, in the Lindenbaum-Tarski algebra of $\Log$; $r$ is derivable in $\Log$ \tiff the quasiidentity $(\phi_1' = 1\; \&  \dots \& \; \phi_n' =1) \Rightarrow \psi' =1$ holds in the quasivariety $\Q_{\Log}$ corresponding to $\Log$, see e.g. \cite{ber88}. Hence, we have the following definition: A quasivariety $\Q$ is \emph{structurally complete} (SC), if  every quasiidentity which holds in  $\mathcal F_{\Q}(\omega)$ also holds in \Q\   see~\cite{ber88,ds16}. \Q\ is \emph{almost  structurally complete} (ASC), if  every active
quasiidentity which holds in  $\mathcal F_{\Q}(\omega)$ also holds in \Q\ , where $(\phi_1' = 1\; \&  \dots \& \; \phi_n' =1) \Rightarrow \psi' =1$ is active if  $\neg (\phi_1' = 1\; \&  \dots \& \; \phi_n' =1)$ does not hold in $\mathcal F_{\Q}(\omega)$,  see~\cite[Section 3]{ds16}.

It is well known see e.g. \cite[Proposition 1.2]{ber88}, that every consequence operation $\vdash$ (or the corresponding logic)  has a unique SC-extension $\vdash '$, that is, $\vdash '$  having the same set of theorems as $\vdash$  such that $\vdash'$ is SC. Note that $\vdash'$ extends $\vdash$ with the rules only. Hence every quasivariety $\Q$ has a unique SC subquasivariety $\Q'$ such that  $\Eq({\Q}) =  \Eq({\Q'})$. Below we describe the SC subquasivariety of \Q,  for those   {\Q} in the paper,  which are ASC. From  Corollary \ref{ASC log.3}  and  from \cite[Corollary 6.6]{DzW11} we have

\begin{corollary}\label{cor:ASC4}
The structurally complete subquasivariety  of {\sf Q}(\ima), respectively, {\sf Q}(\mma), is axiomatized by equations obtained from the axioms of $\mathbf{S4.3DumB_2}$, respectively, of $\mathbf{S4.3B_2}$, and the single quasiidentity  $(\Diamond x \land \Diamond \neg x \; = 1) \:  \Rightarrow \:  0 = 1$.
\end{corollary}
 Observe that $(\Diamond x \land \Diamond \neg x \;= 1) \:  \Rightarrow
\:  0 = 1$ is the algebraic translation of the passive rule $P_2$, see,
e.g.~\cite[p. 530]{DzW11} which was introduced by Rybakov \cite{ryb97}.
Descriptions of admissible rules in  \fma \  and in \gma \  are rather complicated and the statements analogous to the descriptions above of structurally complete subquasivarieties do not hold.  For instance, unification in $\mathbf{S4MB_2}$ and in  $\mathbf{S4MDumB_2}$ is not unitary, hence
not projective, and  $\mathbf{S4MDumB_2}$ is not even almost structurally complete. According to  \cite[4.3.33]{ryb97}, the admissible rules in $\mathbf{S4MDumB_2}$ have no finite basis, indeed, no basis in finitely many variables. Hence a simple description of the structurally complete subquasivariety  of {\sf Q}(\gma) similar to Corollary \ref{cor:ASC4} is not possible.
Instead of $(\Diamond x \land \Diamond \neg x \; = 1) \:  \Rightarrow \:  0 = 1$ infinitely many complicated quasiidentities are needed.


\section{Summary and outlook}

We have investigated varieties of closure algebras of depth two, the canonical relations of which are the identity or the universal relation when restricted to the levels, and have described their associated logics.
We have also discussed the quasivarieties generated by the four classes of algebras and structural completeness. In future work we shall investigate algebras and logics of depth two whose canonical frames are irreflexive, in particular the ones connected to extensions of the provability logic \textbf{GL}.

\end{document}